\documentclass[12pt,dvipsnames]{amsart}

\usepackage{amssymb,mathdots,mathtools,nccmath}

\usepackage[T1]{fontenc}
\usepackage[utf8]{inputenc}
\usepackage[aligntableaux = center]{ytableau}
\usepackage{graphicx}
\usepackage{fullpage}
\usepackage{enumitem}
\usepackage{xcolor}
\usepackage{booktabs}
\usepackage{makecell}
\usepackage{nicematrix}
\usepackage{comment}
\usepackage[normalem]{ulem} 

\usepackage{tikz}
\usetikzlibrary{arrows,decorations.pathreplacing}

\usepackage[sc]{mathpazo}

\usepackage{subcaption}

\newtheorem{thm}{Theorem}[section]
\newtheorem{prop}[thm]{Proposition}
\newtheorem{cor}[thm]{Corollary}
\newtheorem{lem}[thm]{Lemma}

\newtheorem{dfn}[thm]{Definition}
\newtheorem{exm}[thm]{Example}
\newtheorem{rmk}[thm]{Remark}

\DeclareMathOperator{\diag}{diag}
\DeclareMathOperator{\sgn}{sgn}
\newcommand{\sym}[1]{\mathbf{Sym}_{#1}}
\DeclareMathOperator{\asym}{\mathbf{ASym}}

\newcommand{\Symm}[2]{\mathbf{Sym}_{#1_1,...,#1_{#2}}}

\DeclareMathOperator*{\pf}{Pf}

\renewcommand{\le}{\leqslant}
\renewcommand{\ge}{\geqslant}

\newcommand{\F}{\mathsf{F}}

\newcommand{\Mgamma}{M^{\gamma}}
\newcommand{\Mone}{M^{1}}

\newcommand{\Mgammabracket}{M^{\gamma}}
\newcommand{\Mgammarootqbracket}{M^{q^{1/2}\gamma}}

\makeatletter
\providecommand*{\cupdot}{%
  \mathbin{%
    \mathpalette\@cupdot{}%
  }%
}
\newcommand*{\@cupdot}[2]{%
  \ooalign{%
    $\m@th#1\cup$\cr
    \hidewidth$\m@th#1\cdot$\hidewidth
  }%
}
\makeatother

\newcommand{\Def}{\overset{\operatorname{def}}{=}}

\usepackage[colorlinks=true]{hyperref}
\title{Two Littlewood identities for fully inhomogeneous spin Hall-Littlewood symmetric rational functions}
\author[Ilse]{Ilse Fischer}
\thanks{\href{mailto:ilse.fischer@univie.ac.at}{ilse.fischer@univie.ac.at}. The author acknowledges the financial support from the Austrian Science Fund (FWF) grant 10.55776/F100.}

\author[Moritz]{Moritz Gangl}\thanks{\href{mailto:moritz.gangl@univie.ac.at}{moritz.gangl@univie.ac.at}. The author acknowledges the financial support from the Austrian Science Fund (FWF) grant 10.55776/P34931.}
\begin{document} 
\begin{abstract}
    Fully inhomogeneous spin Hall-Littlewood symmetric rational functions $F_\lambda$ arise as partition functions of certain path configurations in the $\mathfrak{sl}_2$ higher spin six vertex models. They are multiparameter generalizations of the classical Hall-Littlewood symmetric polynomials. We establish two new generalizations of the classical Littlewood identity, where we express a weighted sum of  $F_\lambda$'s over all partitions $\lambda$ as a product of the Littlewood kernel and another simple product in one case, and a product of the Littlewood kernel and a Pfaffian in the other case. As a corollary we obtain a novel Littlewood identity for Hall-Littlewood symmetric polynomials.
    
 We also elaborate on the newly established connection between the fully inhomogeneous spin Hall-Littlewood symmetric rational functions $F_\lambda$ and the modified Robbins polynomials, the latter being multivariate generating functions for alternating sign matrices. This connection allowed us to discover the two generalizations of the Littlewood identity and we provide a bijection between the underlying combinatorial models in the case where $\lambda$ is strictly decreasing.
\end{abstract}
\maketitle
\section{Introduction} 
The classical \emph{Littlewood identity} \cite{LittlewoodIdentityOriginal,schur1,schur2} asserts that the sum of \emph{Schur polynomials}  $s_\lambda(x_1,\dots,x_n)$ over all partitions $\lambda$ evaluates to a simple product. 
Concretely, for $n \ge 1$, we have   
\begin{equation}
\label{littlewood} 
    \sum_{\lambda}s_\lambda(x_1,\dots,x_n)=\prod_{i=1}^n\frac{1}{1-x_i}\prod_{1\le i<j\le n}\frac{1}{1-x_ix_j}.
\end{equation}
This identity was already known to Schur~\cite{schur1, schur2}, but combined with the following two variations, the triple was first written down by Littlewood~\cite[p.238]{LittlewoodIdentityOriginal}.
\begin{align}
    \sum_{\text{all parts of $\lambda$ are even}} s_{\lambda}(x_1,\ldots,x_n) &= \prod_{i=1}^n\frac{1}{1-x_i^2}\prod_{1 \le i < j \le n} \frac{1}{1-x_i x_j}\label{eq: second LW id}\\
    \sum_{\text{all parts of $\lambda'$ are even}} s_{\lambda}(x_1,\ldots,x_n) &= \prod_{1 \le i < j \le n} \frac{1}{1-x_i x_j}\label{eq: Littlewood that Gavrilova generalised}
\end{align}
These Littlewood identities lie, besides the \emph{Cauchy identity}~\cite[(4.3)~p.63]{Macdonald}, at the core of algebraic combinatorics and the theory of symmetric functions. The identities have led to several generalizations, variations, and applications throughout combinatorics, representation theory, number theory and integrable probability. For an extensive list in this respect, see, for instance, the introduction of \cite{LWIdentityHistoryFullReferences}.
Notably, all three identities have very beautiful bijective proofs using 
the \emph{Robinson-Schensted-Knuth correspondence.} 

In this paper, we present two generalizations of the Littlewood identity in \eqref{littlewood} for the $\Xi=1$ case of the \emph{fully inhomogeneous spin Hall-Littlewood symmetric rational functions}~$\F_\lambda(u_1,\dots,u_n|s_0,s_1,\ldots)$. These functions and their homogeneous version have appeared in \cite{Borodin, BorodinPetrov, Gavrilova, Petrov} as partition functions of certain square lattice vertex models, where their Yang-Baxter integrability is traced back to the quantum group $U_q(\widehat{\mathfrak{sl}_2})$ and they are also studied from a symmetric function point of view. In \cite{Eigenfunctions2, Eigenfunctions3, Eigenfunctions1} they also appear as eigenfunctions of certain stochastic particle systems. 

For a non-negative integer $n$ and $\lambda=(\lambda_1,\ldots,\lambda_n)$ a partition where we also allow zero parts, the following formula for the $\F_\lambda$'s was shown in \cite{BorodinPetrov}
\begin{equation}\label{eq:def of F lambda}
    \F_\lambda(u_1,\dots,u_n|s_0,s_1,\ldots) = \Symm{u}{n}\left[\prod_{1\le i<j\le n}\frac{u_i-qu_j}{u_i-u_j}\prod_{i=1}^n\left(\frac{1-q}{1-s_{\lambda_i}u_i}\prod_{j=0}^{\lambda_i-1}\frac{u_i-s_j}{1-s_ju_i}\right)\right],
\end{equation}
where 
$$
\Symm{u}{n} G(u_1,\ldots,u_n) \Def \sum_{\sigma \in {\mathfrak S}_n} G(u_{\sigma(1)},\ldots,u_{\sigma(n)}),
$$
using a version of the algebraic Bethe Ansatz.
Note that this is the case of $\xi_j=1$ for all $j\ge 0$ of the polynomials appearing in \cite{BorodinPetrov}, which is a sensible specialization in the case of Littlewood identities and thus will be considered throughout this paper.  

The $\F_\lambda$'s generalize various important symmetric functions. For example, \emph{Schur polynomials} are obtained by setting $q=0$ and $s_j=0$ for all $j$ in $\F_\lambda$. However, they are also generalizations of the \emph{Hall-Littlewood polynomials}, as their name suggests: the Hall-Littlewood polynomials are obtained up to a simple factor when setting only $s_j=0$ for all $j$ in 
$\F_\lambda$, see \eqref{HLspecial}. 

In \cite{Petrov}, Petrov has established a Cauchy identity for the $\F_\lambda$'s that generalizes the classical Cauchy identity.  In this case, the right-hand side does not factorize in general, however, it can be expressed using a determinant. Earlier work in this direction on special cases can be found in \cite{LWCombAspects3,LWCombAspects4,kirillov,warnaar, LWCombAspects13}. In \cite{Gavrilova}, Gavrilova has provided a Littlewood identity connected to Petrov's result for the variation of the classical Littlewood identity given in \eqref{eq: Littlewood that Gavrilova generalised}.
The first main result of this paper is a Littlewood identity, of the type in \eqref{littlewood}, accompanying Petrov's Cauchy identity from \cite{Petrov} mentioned above; see Theorem~\ref{main1}. The other result is another Littlewood identity of the type in \eqref{littlewood} for the fully inhomogeneous spin Hall-Littlewood symmetric rational functions, where this time the right-hand side does factorize, see Theorem~\ref{main2}. 

After a brief remark on notational conventions, we will state the two main results of this paper. For a partition $\lambda$ and a non-negative integer $r$, we denote throughout the paper by $m_r(\lambda)$ the number of parts 
equal to $r$ in the partition $\lambda$. The \emph{$q$-Pochhammer symbol} will also be used, which is defined as
$$
(a;q)_n \overset{\textrm{def}}{=} \prod_{i=0}^{n-1} (1 - a q^{i}). 
$$
Furthermore, we set $\chi_{\mathrm{even}}(n)=1$ if $n$ is even and 
$\chi_{\mathrm{even}}(n)=0$ otherwise. Also, for every triangular array $A=(a_{i,j})_{1 \le i < j \le 2n}$, the \emph{Pfaffian} of the array $A$ is defined as
	\begin{equation}\label{eq: Def of Pfaffian}
		\pf(A)\Def\sum_{\{(i_1,j_1),\dots,(i_n,j_n)\}} \sgn(i_1 j_1 \dots i_n j_n) a_{i_1,j_1} \cdots a_{i_n,j_n},
	\end{equation}
where we sum over all perfect matchings $\{(i_1,j_1),\dots,(i_n,j_n)\}$ of $\{1,\dots,2n\}$ such that $i_1 < \dots < i_n$ and $i_k < j_k$ for all $1 \le k \le n$. 
Our first main result can now be stated as follows.

\begin{thm} 
\label{main1} 
Let $n,p$ be non-negative integers, then
\begin{equation} 
\label{main1:id}
\sum_{\lambda_1 \ge \lambda_2 \ge \ldots \ge \lambda_n \ge 0} 
\prod_{r \ge 0} \frac{(-s_r;q)_{m_r(\lambda)}}{(q;q)_{m_r(\lambda)}} \F_{\lambda}(u_1,\ldots,u_n|s_0,s_1,\ldots)   = 
\prod_{i=1}^{n} \frac{1}{1- u_i}
\prod_{1 \le i < j \le n} \frac{1-q  u_i u_j}{1-u_i u_j} 
\end{equation}
where we set $s_j=s$ for all $j \ge p$ and  $u_i = \frac{s+x_i}{1 + s x_i}$ for all $i$ and a variable $s$, and consider
both sides as power series in $x_1,x_2,\ldots,x_n$.
\end{thm} 
This is a direct generalization of \cite[(1.4)]{bounded}, where the latter is obtained by setting $s_j=-q^{-1}$ and $u_i=\frac{-q^{-1/2}+x_i}{1-q^{-1/2}x_i}$
in $\F_{\lambda}(u_1,\ldots,u_n|s_0,s_1,\ldots)$ and 
$w=-(q+1+q^{-1})$ in the modified Robbins polynomials. The second theorem is a direct generalization of  \cite[Theorem~1.1]{FHRobbins}, which is obtained by setting $\gamma=1$ and $s_j=-q^{-1/2}$.

\begin{thm}
\label{main2}
Let $n,p$ be non-negative integers, then
\begin{multline} 
\sum_{\lambda_1 \ge \lambda_2 \ge \ldots \ge \lambda_n \ge 0} 
\frac{(-\gamma q^{1/ 2};q^{1 / 2})_{m_0(\lambda)}}{(q;q)_{m_0(\lambda)}} (-\gamma^{-1}s_0;q^{1 / 2})_{m_0(\lambda)} 
\\ \times
\prod_{r \ge 1} \frac{(-s_r;q^{1 /2})_{m_r(\lambda)}}{(q^{1 /2};q^{1 /2})_{m_r(\lambda)}}  \F_{\lambda}(u_1,\ldots,u_n|s_0,s_1,\ldots) 
=  \prod_{i=1}^n \frac{1+q^{1/2}}{1-u_i}
\prod_{1 \le i < j \le n} \frac{1-q  u_i u_j}{u_i-u_j} \\
\times 
\pf_{\chi_{\mathrm{even}}(n) \le i < j \le n} \left(
		\begin{cases}
			1 + (\gamma-1) \frac{(q^{1/2} - \gamma^{-1} s_0)(1-u_j)}{(1+q^{1/2}) (1- s_0 u_j)}, & i=0,\\
			\frac{(u_i-u_j)((1+q)(1- q^{1 /2} u_i u_j) + (u_i + u_j)(q^{1 /2}- q)  +
			(\gamma-1) f(u_i,u_j))}{(1+q^{1/2})(1- u_i u_j)(1- q u_i u_j)},& i\ge 1,
		\end{cases} \right) 
\end{multline}
where 
\begin{multline*} 
f(u,v) = 
\frac{(q^{1/2} - \gamma^{-1} s_0) (1-u v)}{(1+q^{1/ 2})(1-s_0 u) (1-s_0 v)}
\left( (1+ q^{1 /2})(1+ \gamma^{-1} s_0) (1 + q \gamma u v) \right. \\
\left. +(1 + \gamma)(q - \gamma^{-1} s_0)(1 - q^{1/ 2} u v) - (1 + \gamma) q^{1/ 2} (q^{1/ 2} + \gamma^{-1} s_0) (u + v) \right), 
\end{multline*} 	
and we set $s_j=s$ for all $j \ge p$,  $u_i = \frac{s+x_i}{1 + s x_i}$ for all $i$ and a variable $s$, and consider
both sides as power series in $x_1,x_2,\ldots,x_n$.
\end{thm}
Kawanaka's identity for \emph{Hall-Littlewood polynomials} \cite{KawankasIdentity} is a special case of this identity too. It is obtained by first setting $s_0=0$, then $\gamma=0$, and then $s_j=0$ for all $j>0$.

Setting $\gamma=1$, we obtain the following corollary. In fact, Corollary~\ref{main2:cor} plays an essential role 
in the proof of Theorem~\ref{main2} as it will be necessary to prove this before we can deduce Theorem~\ref{main2} from it.

\begin{cor}\label{main2:cor}
    Let $n,p$ be non-negative integers, then 
    \begin{multline}\label{eq: main2 cor} 
    \sum_{\lambda_1 \ge \lambda_2 \ge \ldots \ge \lambda_n \ge 0} 
    \prod_{r \ge 0} \frac{(-s_r;q^{1 /2})_{m_r(\lambda)}}{(q^{1 /2};q^{1 /2})_{m_r(\lambda)}}  \F_{\lambda}(u_1,\ldots,u_n|s_0,s_1,\ldots)
    =  \prod_{i=1}^n \frac{1+q^{1/2}}{1-u_i}\\
    \times
    \prod_{1 \le i < j \le n} \frac{1-q  u_i u_j}{u_i-u_j} 
    \pf_{\chi_{\mathrm{even}}(n) \le i < j \le n} \left(
    		\begin{cases}
    			1, & i=0,\\
    			\frac{(u_i-u_j)((1+q)(1- q^{1 /2} u_i u_j) + (u_i + u_j)(q^{1 /2}- q)) }{(1+q^{1/2})(1- u_i u_j)(1- q u_i u_j)},& i\ge 1,
    		\end{cases} \right),
    \end{multline}
where we set $s_j=s$ for all $j \ge p$,  $u_i = \frac{s+x_i}{1 + s x_i}$ for all $i$ and a variable $s$, and consider
both sides as power series in $x_1,x_2,\ldots,x_n$.    
\end{cor}

As mentioned above, the fully inhomogeneous spin Hall-Littlewood symmetric rational functions are a generalization of \emph{Hall-Littlewood polynomials} $P_\lambda(x_1,\dots,x_n;q)$, see \cite{BorodinPetrov, Petrov}. Recall that the Hall-Littlewood polynomials are given by the following bialternant formula.
\begin{equation*}
P_\lambda(x_1,\dots,x_n;q)\Def (1-q)^n\prod_{r\ge0}\frac{1}{(q;q)_{m_r(\lambda)}} \Symm{x}{n}{\left[\prod_{1\le i<j\le n} \frac{x_i-qx_j}{x_i-x_j} \prod_{i=1}^nx_i^{\lambda_i}\right]}
\end{equation*} 
Using \eqref{eq:def of F lambda}, it follows that 
\begin{equation}
\label{HLspecial} 
     \F_{\lambda}(u_1,\ldots,u_n|0)=\prod_{r\ge 0}(q;q)_{m_r(\lambda)}P_\lambda(x_1,\dots,x_n;q).
\end{equation}
Therefore, setting $s_j=0$ for all $j\ge 0$ in Corollary \ref{main2:cor}, we obtain the following Littlewood identity for Hall-Littlewood polynomials, which, to the best of our knowledge, is new.
\begin{cor}
    For any non-negative integer $n$, we have
    \begin{multline}
        \sum_{\lambda_1 \ge \lambda_2 \ge \ldots \ge \lambda_n \ge 0}\prod_{r\ge 0}(-q^{1/2};q^{1/2})_{m_r(\lambda)}P_\lambda(x_1,\dots,x_n;q)=\prod_{i=1}^n\frac{1+q^{1/2}}{1-x_i}\prod_{1\le i<j\le n}\frac{1-qx_ix_j}{x_i-x_j}\\\times \pf_{\chi_{\mathrm{even}}(n) \le i < j \le n} \left(
    		\begin{cases}
    			1, & i=0,\\
    			\frac{(x_i-x_j)((1+q)(1- q^{1 /2} x_i x_j) + (x_i + x_j)(q^{1 /2}- q)) }{(1+q^{1/2})(1- x_i x_j)(1- q x_i x_j)},& i\ge 1,
    		\end{cases} \right).
    \end{multline}
   \end{cor}

\medskip

Next, we briefly outline how our results were discovered. In \cite[Section~8]{FHRobbins}, the following connection was observed between fully inhomogeneous spin Hall-Littlewood symmetric rational functions and the \emph{modified Robbins polynomials} $R^*_{(k_1,\dots,k_n)}(x_1,\dots,x_n;u,v,w)$. The modified Robbins polynomials are generating functions of \emph{alternating sign matrices} as will be recalled in detail in Section~\ref{Sec: combinatorics}.

\begin{lem}\label{lem: Connection Robbins sHL}
    For any integer $n\ge 1$ and a partition $\lambda=(\lambda_1,\dots,\lambda_n)$, where parts equal to zero are allowed, the following relation between the fully inhomogeneous spin Hall-Littlewood symmetric rational functions and the modified Robbins polynomial holds.
    \begin{multline}\label{eq: relation sHL and mod Robbins}
    R^*_{(\lambda_n,\dots,\lambda_1)}(x_1,\dots,x_n;q^{1/2}-q^{-1/2},q^{1/2}-q^{-1/2},q^{-1}-q)
     =(1-q^{-1})^{\binom{n}{2}}\\\left.\times\prod_{i=1}^n\frac{1-q^{-1}}{1-q+(q^{1/2}-q^{-1/2})x_i} \F_\lambda(u_1,\dots,u_n|-q^{-1/2},-q^{-1/2},\ldots)\right|_{u_i=\frac{-q^{-1/2}+x_i}{1-q^{-1/2}x_i}}.
\end{multline}
\end{lem}

The significance of the modified Robbins polynomials stems from the fact that, in recent work relating alternating sign triangles as well as \emph{alternating sign trapezoids} to \emph{totally symmetric self-complementary plane partitions} and \emph{descending plane partitions} by rather heavy computations, certain Littlewood identities for these polynomials play the most important role, see \cite[Lemma~4.1]{Fis19b}, \cite[Lemma~9]{Fis19a}, \cite[Lemma~3.4]{AlternatingSignPentagonsAndMagogPentagons} 
and \cite[Lemma~9]{fourfold}. This also motivates our study of these identities: bijective proofs of the generalized Littlewood identities would be an important step toward establishing bijections between alternating sign matrix objects and plane partition objects. Moreover, if such bijective proofs exist, it is conceivable that the Robinson–Schensted–Knuth correspondence—which provides a bijective proof of the ordinary Littlewood identity—offers a natural starting point.

Although Lemma \ref{lem: Connection Robbins sHL} can easily be deduced from the bialternant formulas for the respective symmetric functions, we also provide a bijective proof for this relation in the case when $\lambda$ is strictly decreasing, after we have introduced the combinatorial models for the polynomials  $\F_\lambda(u_1,\dots,u_n|s_0,s_1,\ldots)$ as well as the modified Robbins polynomials in Section~\ref{Sec: combinatorics}. A feature of the case where $\lambda$ is strict is that the combinatorial models are signless for both polynomials. This new connection between modified Robbins polynomials and fully inhomogeneous spin Hall-Littlewood symmetric rational functions allowed H\"ongesberg and the first author of this paper to conjecturally generalize two Littlewood identities for the modified Robbins polynomials to the level of fully inhomogeneous spin Hall-Littlewood symmetric rational functions. As announced together with the conjectures in \cite{FHRobbins}, the proofs of the two identities now follow in this paper.

\medskip

{\bf Outline of the paper.} In Section \ref{Sec: combinatorics} we provide some insights concerning the combinatorics underlying our main results. We start by introducing the $\mathfrak{sl}_2$~higher spin six vertex model as the combinatorial model for the fully inhomogeneous spin Hall-Littlewood symmetric rational functions in Subsection \ref{Subsec: higher spin 6VM}. We then explain the background for the modified Robbins polynomial in Subsection \ref{Subsec: mod Robbins Polynomials} and end the section by elaborating more on the connection hinted at in Lemma \ref{lem: Connection Robbins sHL} in Subsection \ref{Subsec: Connection to Robbins polynomials}. 

Miraculously, Theorem \ref{main1} and Theorem \ref{main2} admit fairly similar proofs. We use a recurrence relation for fully inhomogeneous spin Hall-Littlewood symmetric rational functions that is established in Section \ref{Sec: recurrence relations} to deduce recurrence relations for the left-hand sides of the identities in both theorems. We then use induction to show in Sections~\ref{sec: proof of Thm main1} and \ref{sec: proof of Thm main2}
that the right-hand sides of Theorems~\ref{main1} and  \ref{main2} satisfy the same recurrence relations. This is done by reducing the equations that are needed to the polynomial identities stated in Lemma \ref{lem: Key lemma for main1} and Lemma \ref{lem: Key lemma for main2}, respectively, which we then prove by 
interpolation.

\section{The combinatorics of the relation between fully inhomogeneous spin Hall-Littlewood symmetric rational functions and alternating sign matrices}\label{Sec: combinatorics}
This section is structured as follows. We recall the combinatorial model for fully inhomogeneous spin Hall-Littlewood symmetric rational functions in Subsection \ref{Subsec: higher spin 6VM}. Then we explain the background material on the modified Robbins polynomial in Subsection~\ref{Subsec: mod Robbins Polynomials}, and end the section by providing a bijective proof of Lemma \ref{lem: Connection Robbins sHL} in the case of a strictly decreasing partition in Subsection \ref{Subsec: Connection to Robbins polynomials}.

\subsection{The $\mathfrak{sl}_2$~higher spin six vertex model}\label{Subsec: higher spin 6VM} Fully inhomogeneous spin Hall-Littlewood symmetric rational functions are, in a sense, the generating functions of the $\mathfrak{sl}_2$~\emph{higher spin six vertex model} with certain boundary conditions \cite{BorodinPetrov, Petrov}, which we introduce next.

Let $n\ge 0$ be an integer and $\lambda=(\lambda_1,\dots,\lambda_n)$ be a partition of length $n$ where we also allow parts equal to zero in $\lambda$. We consider ensembles of $n$ paths 
\begin{itemize} 
\item with step set $\{(1,0),(0,1)\}$ in $\mathbb{Z}_{\ge 0}\times \{1,\dots,n\}$,
\item such that
the $i$-th path starts at vertex $(0,i)$ and ends at vertex $(\lambda_i,n)$, $i=1,\ldots,n$, 
\item and the paths are mutually non-crossing and must not intersect along horizontal steps $(1,0)$, however, they may intersect along vertical steps $(0,1)$. 
\end{itemize} 
An example for the partition $\lambda=(5,5,2,0)$ is given in Figure~\ref{fig: A set of paths in the higher spin six vertex model}.
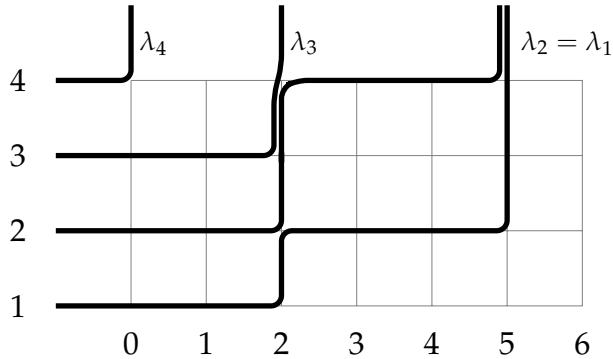
\begin{figure}[h]
    \centering
    \begin{tikzpicture}
        \draw[step=1cm,gray,very thin] (0,1) grid (6,4);
        \path [draw, line width=2pt, rounded corners](-1,4)--(0,4)--(0,5) (-1,3)--(1.9,3)--(1.9,3.8)--(2,4.2)--(2,5) (-1,2)--(2,2)--(2,2.9)--(2,3.05)--(2,3.9)--(2.2,4)--(4.9,4)--(4.9,5) (-1,1)--(2,1)--(2,2)--(5,2)--(5,5);
        \node at (0.3,4.5) {\Small$\lambda_4$};
        \node at (2.3,4.5) {\Small$\lambda_3$};
        \node at (5.8,4.5) {\Small$\lambda_2=\lambda_1$};

        \node at (-1.5,1) {$1$};
        \node at (-1.5,2) {$2$};
        \node at (-1.5,3) {$3$};
        \node at (-1.5,4) {$4$};

        \node at (1,0.5) {$1$};
        \node at (2,0.5) {$2$};
        \node at (3,0.5) {$3$};
        \node at (4,0.5) {$4$};
        \node at (5,0.5) {$5$};
        \node at (6,0.5) {$6$};
        \node at (0,0.5) {$0$};
    \end{tikzpicture}
    \caption{An ensemble of paths in the $\mathfrak{sl}_2$~higher spin six vertex model for $\lambda=(5,5,2,0)$ with additional horizontal steps at the start and additional vertical steps at the top.}
    \label{fig: A set of paths in the higher spin six vertex model}
\end{figure}

For the purpose of introducing the weight, we extend each path at its starting point by a horizontal unit step to the left and at its endpoint by a 
vertical step towards the top. Now we assign weights to the individual vertices of the ensemble depending on the local configurations, see Figure \ref{fig:higher spin six vertex model weigths} for all possible configurations. The total weight of the ensemble is then the product of the weights of all the vertices in $\mathbb{Z}_{\ge 0}\times \{1,\dots,n\}$. 

The individual vertex weight in position $(j,i)$  depends on a global parameter $q$, on the \emph{spectral parameter} $u_i$, and the \emph{spin parameter} $s_j$. The local configurations are encoded by four integers
$i_1,j_1,i_2,j_2$, where $i_1,i_2$ denote the numbers of paths which enter from below and exit towards the top, respectively, and $j_1,j_2$ denote the numbers of paths which enter from the left and exit towards the right, respectively. By definition $j_1,j_2 \in \{0,1\}$ and $i_1+j_1=i_2+j_2$, where the latter is also addressed as \emph{arrow preservation}. The concrete vertex weights $w_{u,s}(i_1,i_2;j_1,j_2)$ are provided in Figure \ref{fig:higher spin six vertex model weigths}, where we have replaced $u_i,s_j$ by generic variables $u,s$.\footnote{These are the weights introduced in \cite{BorodinPetrov, Petrov} for the special case $\xi_i=1$, but they generalize the weights from \cite{Borodin} where all $s_j$ are equal.}



\begin{figure}[h]
    \centering
    \begin{tabular}{c||c|c|c|c}
         \begin{tikzpicture}[scale=0.5]
             \path [draw=black!40, line width=3pt] (-1,0)--(1,0) (0,-1)--(0,1);
             \filldraw[black] (0,0) circle (2pt);
             \node at (0,-1.5) {$i_1$};
             \node at (0,1.5) {$i_2$};
             \node at (-1.5,0) {$j_1$};
             \node at (1.5,0) {$j_2$};
         \end{tikzpicture}&\begin{tikzpicture}[scale=0.5]
             \path[draw=black, dotted] (-1,0)--(1,0) (0,1)--(0,-1);
             \draw (0,-1)--(0,-0.2) (0.2,-1)--(0.2,-0.05) (-0.2,-1)--(-0.2,-0.05) (0,1)--(0,0.2) (0.2,1)--(0.2,0.05) (-0.2,1)--(-0.2,0.05);
             \filldraw[black] (0,0) circle (2pt);
             \node at (0,-1.5) {$g$};
             \node at (0,1.5) {$g$};
         \end{tikzpicture}&\begin{tikzpicture}[scale=0.5]
             \path[draw=black, dotted] (-1,0)--(1,0) (0,1)--(0,-1);
             \draw (0,-1)--(0,-0.2) (0.2,-1)--(0.2,-0.05) (-0.2,-1)--(-0.2,-0.05) (0,1)--(0,0.2)  (-0.2,1)--(-0.2,0.05) (0.2,0)--(1,0);
             \filldraw[black] (0,0) circle (2pt);
             \node at (0,-1.5) {$g$};
             \node at (0,1.5) {$g-1$};
         \end{tikzpicture}&\begin{tikzpicture}[scale=0.5]
             \path[draw=black, dotted] (-1,0)--(1,0) (0,1)--(0,-1);
             \draw (0,-1)--(0,-0.2) (0.2,-1)--(0.2,-0.05) (-1,0)--(-0.2,0) (0,1)--(0,0.2) (0.2,1)--(0.2,0.05) (-0.2,1)--(-0.2,0.05);
             \filldraw[black] (0,0) circle (2pt);
             \node at (0,-1.5) {$g$};
             \node at (0,1.5) {$g+1$};
         \end{tikzpicture}&\begin{tikzpicture}[scale=0.5]
             \path[draw=black, dotted] (-1,0)--(1,0) (0,1)--(0,-1);
             \draw (0,-1)--(0,-0.2) (0.2,-1)--(0.2,-0.05) (-0.2,-1)--(-0.2,-0.05) (0,1)--(0,0.2) (0.2,1)--(0.2,0.05) (-0.2,1)--(-0.2,0.05) (-1,0)--(-0.2,0) (0.2,0)--(1,0);
             \filldraw[black] (0,0) circle (2pt);
             \node at (0,-1.5) {$g$};
             \node at (0,1.5) {$g$};
         \end{tikzpicture} \\
         \hline
         $w_{u,s}(i_1,i_2;j_1,j_2)$&\Large$\frac{1-suq^{g}}{1-su}$&\Large$\frac{u(1-s^2q^{g-1})}{1-su}$&\Large$\frac{1-q^{g+1}}{1-su}$&\Large$\frac{u-sq^g}{1-su}$
    \end{tabular}
    \caption{$\mathfrak{sl}_2$ higher spin six vertex model weights}
    \label{fig:higher spin six vertex model weigths}
\end{figure}
Using the integrability of this vertex model and a version of the algebraic Bethe Ansatz, it was shown in \cite{BorodinPetrov} that the generating function $\F_\lambda(u_1,\dots,u_n| s_0,s_1,\dots)$ of all such ensembles of paths for a fixed partition $\lambda$ is given by the bialternant formula stated in the introduction in \eqref{eq:def of F lambda}.

This path model can also be used to see combinatorially that $\F_\lambda(u_1,\dots,u_n| s_0,s_1,\dots)$ specializes to Schur polynomials after setting $q=0$ and $s_j=0$ for all $j\in\mathbb{Z}_{\ge 0}$: the weights can then be interpreted so that horizontal steps at height $i$ are assigned the weight $u_i$ while vertical steps are assigned the weight $1$. Then, by stretching the paths horizontally apart, one obtains non-intersecting lattice paths which are known to be a combinatorial model for Schur polynomials, see  \cite{EC2}.

\subsection{The modified Robbins polynomials}\label{Subsec: mod Robbins Polynomials}
We provide a brief summary of various results on Robbins polynomials and their connections to alternating sign matrices, as well as symmetric functions, which are necessary to explain the connection to the $\mathfrak{sl}_2$~higher spin six vertex model. The main focus lies on introducing the combinatorial model of \emph{down-arrowed monotone triangles} for the modified Robbins polynomials.

A \emph{Gelfand-Tsetlin pattern} is a triangular array of integers of the following form 
$$\begin{matrix}
     & & & &a_{1,1}& & & &\\
     & & &a_{2,1}& &a_{2,2}& & &\\
     & & \dots& & \dots & & \dots& &\\
     &a_{n-1,1}& & \dots& & \dots& &a_{n-1,n-1}&\\
     a_{n,1}& &a_{n,2}& & \dots& &a_{n,n-1}& &a_{n,n}
    \end{matrix},$$
with weak increase along $\nearrow$-diagonals and $\searrow$-diagonals. The most commonly used combinatorial model for Schur polynomials is certainly \emph{semistandard Young tableaux}; however, another 
possibility are Gelfand-Tsetlin patterns: assign the following weight to a Gelfand-Tsetlin pattern with $n$ rows 
\begin{equation}
\label{eq: polynomial weight of GTs} 
    \prod_{i=1}^{n} x_i^{(\text{sum of elements in row $i$})-(\text{sum of elements in row $i-1$)}},
\end{equation}
then $s_{\lambda}(x_1,\ldots,x_n)$ is the generating function of Gelfand-Tsetlin patterns where the bottom row is the reversed partition $\lambda$ and we add trailing zeros to $\lambda$ if necessary to have $n$ parts in $\lambda$. The bijection between Gelfand-Tsetlin patterns and semistandard Young tableaux is simple and works as follows: the shape of the entries less than or equal to $i$ in the semistandard Young tableau is the $i$-th row of the Gelfand-Tsetlin pattern in reverse order. 

Gelfand-Tsetlin patterns with strictly increasing rows are called \emph{monotone triangles}. Monotone triangles with 
bottom row $1,2,\ldots,n$ are in easy bijective correspondence with $n \times n$ alternating sign matrices, see for instance \cite{Bre99}.
To sketch the bijection, we note that the numbers in the $i$-th row of the monotone triangle
are just the positions of the $1$'s in the $01$-string we obtain when adding the first $i$ rows of the ASM. 

We introduce a certain polynomial weight on monotone triangles from \cite{nASMDPP}. We call an entry of a monotone triangle
\begin{itemize}
\item \emph{left-leaning} if the entry is equal to its $\swarrow$-neighbour,
\item \emph{right-leaning} if the entry is equal to its $\searrow$-neighbour,
\item \emph{special} if it is neither left nor right leaning.
\end{itemize}
Denote the number of such entries in row $i$ of a given monotone triangle $M$ by $l_i(M)$, $r_i(M)$ and $s_i(M)$ respectively, where we set $l_0(M)=r_0(M)=s_0(M)=0$. Furthermore, we denote 
\begin{equation*}
    d_i(M)=(\textrm{sum of entries in row }i) - (\textrm{sum of entries in row }i-1)+r_{i-1}(M)-l_{i-1}(M).
\end{equation*}
We assign the following polynomial weight to a given monotone triangle $M$. 
\begin{equation}\label{eq: polynomial weight of MTs}
    \prod_{i=1}^nu^{r_{i-1} (M)}v^{l_{i-1}(M)} (w+ux_i+vx_i^{-1})^{s_{i-1}(M)} x_i^{d_i(M)}
\end{equation}
Note the close similarity to the weights in \eqref{eq: polynomial weight of GTs}.

Next, we define certain decorated monotone triangles. The definition first appeared in \cite{nASMDPP}, and it is not hard to see that these objects with a given bottom row have the same generating function as monotone triangles with the same given bottom row and weight given by \eqref{eq: polynomial weight of MTs}.

\begin{dfn}\label{DefinitionDAMTUndModifiedRobbins}
A \emph{down arrowed monotone triangle (DAMT)} is a monotone triangle where each entry not in the bottom row carries a decoration from $\{\swarrow, \downarrow, \searrow\}$ such that the following two conditions are satisfied.
\begin{itemize}
\item If the entry is left-leaning, then it must carry $\swarrow$.
\item If the entry is right-leaning, then it must carry $\searrow$.
\end{itemize}
We assign the following weight to a given DAMT.
\begin{equation}
u^{\#\searrow}v^{\#\swarrow}w^{\#\downarrow}\prod_{i=1}^{n}x_{i}^{(\textrm{sum of entries in row }i)-(\textrm{sum of entries in row }i-1)+(\# \searrow \textrm{ in row i-1})-(\# \swarrow \textrm{ in row i-1})}
\end{equation}
The \emph{modified Robbins polynomial} $R^{*}_{(k_{1},\dots,k_{n})}(x_{1},\dots,x_{n};u,v,w)$ is the generating function 
of DAMTs with bottom row $k_1,\ldots,k_n$ with respect to the weight just defined, see \cite[Theorem~3.1]{nASMDPP}.
\end{dfn}

\begin{exm}
In the following, we display a monotone triangle with polynomial weight 
$u^3(w+ux_3+vx_3^{-1})(w+ux_4+vx_4^{-1})^2x_1^5x_2^5x_3^5x_4^5$ on the left, and a down arrowed monotone triangle on the right with weight $u^3vw^2x_1^5x_2^5x_3^5x_4^4$. 
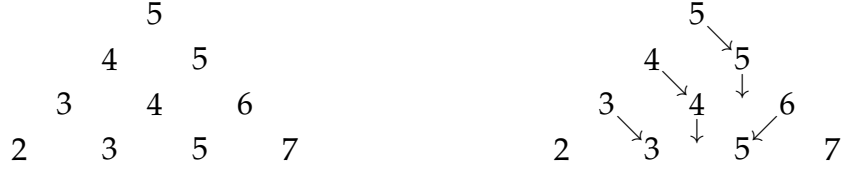
\begin{center}
    \begin{tikzpicture}[scale=0.6]
    \node at (1,1) {$2$};
    \node at (3,1) {$3$}; 
    \node at (5,1) {$5$};
    \node at (7,1) {$7$};
    \node at (2,2) {$3$};
    \node at (4,2) {$4$};
    \node at (6,2) {$6$};
    \node at (3,3) {$4$};
    \node at (5,3) {$5$};
    \node at (4,4) {$5$};
       \end{tikzpicture}\hspace{3cm} 
        \begin{tikzpicture}[scale=0.6]
    \node at (1,1) {$2$};
    \node at (3,1) {$3$};
    \node at (5,1) {$5$};
    \node at (7,1) {$7$};
    \node at (2,2) {$3$};
    \node at (4,2) {$4$};
    \node at (6,2) {$6$};
    \node at (3,3) {$4$};
    \node at (5,3) {$5$};
    \node at (4,4) {$5$};

    \node at (2.5,1.45) {\footnotesize $\searrow$};
    \node at (4.5,3.45) {\footnotesize $\searrow$};
    \node at (3.5,2.45) {\footnotesize $\searrow$};
    \node at (5.5,1.45) {\footnotesize $\swarrow$};

    \node at (5,2.4) {\footnotesize $\downarrow$};
    \node at (4,1.4) {\footnotesize $\downarrow$};

    \end{tikzpicture}
\end{center}
\end{exm}

The modified Robbins polynomials generalize the Schur polynomials as
\begin{equation}
\label{schur}  
R^{*}_{(k_{1},\dots,k_{n})}(x_{1},\dots,x_{n};0,0,1) = \prod_{i=1}^n x_i^{n-1} s_{(k_n-2n+2,k_{n-1}-2n+4,\ldots,k_1)}(x_1,\ldots,x_n).
\end{equation} 
Indeed, when setting $(u,v,w)=(0,0,1)$, we can assume that all entries are decorated with $\downarrow$, and thus all $\nearrow$-diagonals as well as all $\searrow$-diagonals are strictly increasing. Subtracting $i-1$ from the $i$-th $\searrow$-diagonal from the left for all $i$ and then subtracting $i-1$ from the $i$-th $\nearrow$-diagonal from the left for all $i$ establishes a bijection with Gelfand-Tsetlin patterns with bottom row $k_1,k_2-2,\ldots,k_n-2n+2$, and an analysis of the change of weight establishes the identity. 

Next we see that a certain evaluation of the modified Robbins polynomials gives the number of $n \times n$ alternating sign matrices.

\begin{prop}\label{Different Evaluations and alternating sign matrices} 
The evaluation $R^{*}_{(k_{1},\ldots,k_{n})}(1,\dots,1;1,1,-1)$ is the number of monotone triangles with bottom 
row $k_1,\ldots,k_n$. In particular, $R^{*}_{(1,2,\ldots,n)}(1,\ldots,1;1,1,-1)$ is the number of $n \times n$ alternating sign matrices.
\end{prop} 

\begin{proof} 
Given a monotone triangle with bottom row $k_1,\ldots,k_n$, we determine all possible decorations. If the entry is left-leaning or right-leaning, then there is only one possible decoration. In this case, the contribution of the decoration to the weight is $u=1$ or $v=1$ under the specialization we consider. For all other entries, all three decorations are possible, and the total contribution of these decorations to the weight is also $u+v+w=1+1+(-1)=1$.
\end{proof} 

In \cite{nASMDPP}, the following bialternant formula for the modified Robbins polynomials was established using a recursive approach. It is a generalization of the bialternant formula for Schur polynomials.

\begin{thm} 
For any $n \ge 1$ and any strictly increasing sequence $k_1,\ldots,k_n$, we have 
\begin{equation}\label{eq:Antisymm for mod Robbins}
    R^{*}_{(k_{1},\dots,k_{n})}(x_{1},\dots,x_{n};u,v,w)=\frac{\asym_{x_{1},\dots,x_{n}}\left[\prod_{1\le i<j\le n}(u x_i x_j + v   + w x_i)\prod_{i=1}^{n}x_{i}^{k_{i}}\right]}{\prod_{1\le i<j\le n}(x_{j}-x_{i})}, 
\end{equation}
where we denote by $\asym_{x_1,\ldots,x_n} F(x_{1},\dots,x_{n})\Def\sum\limits_{\sigma \in {\mathfrak S}_n}\sgn(\sigma) F(x_{\sigma(1)},\dots,x_{\sigma(n)})$ the antisymmetrizer.
\end{thm} 

Note that the modified Robbins polynomial can also be written in terms of the symmetrizer as follows.
$$
R^{*}_{(k_{1},\dots,k_{n})}(x_{1},\dots,x_{n};u,v,w)= \Symm{x}{n} \left[\prod_{1\le i<j\le n} \frac{u x_i x_j + v   + w x_i}{x_j-x_i} \prod_{i=1}^{n}x_{i}^{k_{i}}\right]
$$

\begin{rmk}\label{Verbindung Robbins and modified Robbins}
Let us note that the \emph{ordinary (non-modified) Robbins polynomials} were also defined in \cite{nASMDPP} and they are given by 
\begin{multline}\label{Non Modified Robbins Polynomial Definition}
    R_{(k_{1},\dots,k_{n})}(x_{1},\dots,x_{n};t,u,v,w) \\\Def \frac{\asym_{x_1,\ldots,x_n} \left[ \prod_{1 \le i \le j \le n}  (tx_{j}+u x_i x_j + v   + w x_i) \prod_{i=1}^{n} x_i^{k_i-1} \right]}
{\prod_{1 \le i < j \le n} (x_j-x_i)}.
\end{multline}
The three differences are as follows: (i) in the product over all $(i,j)$ with 
$1 \le i \le j \le n$, we also admit $i=j$, (ii) we have the additional parameter $t$ and (iii) the exponents of the $x_i$'s are shifted by minus $1$. Robbins polynomials are generating functions of arrowed Gelfand-Tsetlin patterns, which are decorated Gelfand-Tsetlin patterns where decorations are arrows that are 
pointing up. In this case, one can extend the definition to integer sequences $k_1,\ldots,k_n$ that are not necessarily 
strictly increasing at the cost of involving a sign.  

Note that upon setting $s_j=s$ for all $j\ge 0$ and then substituting $u_i=\frac{x_i+s}{1+s x_i}$ in \eqref{eq:def of F lambda} we obtain the following relation between the fully inhomogeneous spin Hall-Littlewood symmetric rational functions and the ordinary Robbins polynomials.
\begin{multline}
    R_{(\lambda_n+1,\dots,\lambda_1+1)}(x_1,\dots,x_n;1-s^2q,s(1-q),s(1-q),s^2-q) =\frac{(1-s^2)^{\binom{n+1}{2}}}{(1-q)^n}\\\left.\times \prod_{i=1}^n\frac{(1-s^2q)x_{i}+s(1-q) (x_i^2+1)  + (s^2-q) x_i}{1+sx_i} \F_\lambda(u_1,\dots,u_n|s,s,\ldots)\right|_{u_i=\frac{x_i+s}{1+s x_i}}
\end{multline}
This would also allow us to specialize the new Littlewood identities for the fully inhomogeneous spin Hall-Littlewood polynomials from Theorem \ref{main1} and Theorem \ref{main2} into Littlewood identities for ordinary Robbins polynomials with $u=v$ too.

Also note that the Robbins polynomials generalize the \emph{Hall-Littlewood polynomials}, as shown by the following relations.
\begin{equation}
    P_{\lambda}(x_{1},\dots,x_{n};t)=  (-1)^{\binom{n}{2}} \prod_{r\ge 0} \frac{1}{(t;t)_{m_r(\lambda)}} R_{\lambda}(x_{1},\dots,x_{n};-t,0,0,1)
\end{equation}
\begin{equation}
    P_{\lambda}(x_{1},\dots,x_{n};t)=\prod_{r\ge 0} \frac{1}{(t;t)_{m_r(\lambda)}} R_{(\lambda_n,\dots,\lambda_1)}(x_{1},\dots,x_{n};1,0,0,-t)
\end{equation}
\end{rmk}

\subsection{Modified Robbins polynomials are special cases of fully inhomogeneous spin-Hall Littlewood symmetric rational functions combinatorially}\label{Subsec: Connection to Robbins polynomials}


We prove Lemma \ref{lem: Connection Robbins sHL} bijectively for a partition with strictly decreasing parts $\lambda=(\lambda_1>\dots >\lambda_n\ge 0)$
and develop some combinatorial understanding of this connection between the $\mathfrak{sl}_2$~higher spin six vertex model and DAMTs.

For the bijective proof of \eqref{eq: relation sHL and mod Robbins},
we first note that $w_{u,-q^{-1/2}}(2,1;0,1)$ vanishes, see Figure~\ref{fig:higher spin six vertex model weigths}. Hence, no north step can be contained in precisely two paths, since the top boundary condition prescribes $n$ different north steps for the $n$ paths, and so the paths would have to separate again, and hence a configuration of the form $(2,1;0,1)$ would have to appear.

Furthermore, no north step can be contained in more than two paths, because otherwise there would also have to be a configuration of the type $(2,1;0,1)$ appearing in the ensemble, as the fact that vertical steps must not intersect implies that from each horizontal line to the next, only one path can separate from a set of paths that previously shared a vertical step. 

Therefore, we are left with the table in Figure~\ref{fig:weights if s is minus root of q} of the five configurations to which we can restrict without loss of generality and their weights expressed in the generic variable of the $\mathfrak{sl}_2$~higher spin six vertex model $u$ in the second row and in the generic DAMT variable $x=\frac{u+q^{-1/2}}{1+q^{-1/2}u}$ in the third row, where we set the DAMT parameter $v=q^{1/2}-q^{-1/2}$. Note that here $u$ does not have the same meaning as in the definition of the modified Robbins polynomials, whereas $v$ corresponds in the end to the $v$ in the definition of the Robbins polynomials. 

\begin{figure}[h]
    \centering
    \begin{tabular}{c||c|c|c|c|c}
    \begin{tikzpicture}[scale=0.35]
             \path [draw=black!40, line width=3pt] (-1,0)--(1,0) (0,-1)--(0,1);
             \filldraw[black] (0,0) circle (2pt);
             \node at (0,-1.5) {\Small$i_1$};
             \node at (0,1.5) {\Small$i_2$};
             \node at (-1.5,0) {\Small$j_1$};
             \node at (1.5,0) {\Small$j_2$};
         \end{tikzpicture}&
         \begin{tikzpicture}[scale=0.7]
         \path [draw, dotted] (-1,0)--(1,0) (0,-1)--(0,1);
             \path [draw, line width=2pt] (0,-1)--(0,-0.2) (0,0.2)--(0,1);
             \filldraw[red] (0,0) circle (3pt);
         \end{tikzpicture}& \begin{tikzpicture}[scale=0.7]
         \path [draw, dotted] (-1,0)--(1,0) (0,-1)--(0,1);
             \path [draw, line width=2pt] (-1,0)--(-0.2,0) (0.2,0)--(1,0);
             \filldraw[red] (0,0) circle (3pt);
         \end{tikzpicture}&\begin{tikzpicture}[scale=0.7]
         \path [draw, dotted] (-1,0)--(1,0) (0,-1)--(0,1);
             \path [draw, line width=2pt] (0,-1)--(0,-0.2) (0,0.2)--(0,1) (-1,0)--(-0.2,0) (0.2,0)--(1,0);
             \filldraw[red] (0,0) circle (3pt);
         \end{tikzpicture}&\begin{tikzpicture}[scale=0.7]
         \path [draw, dotted] (-1,0)--(1,0) (0,-1)--(0,1);
             \path [draw, line width=2pt] (0,-1)--(0,-0.2) (0.2,0)--(1,0);
             \filldraw[red] (0,0) circle (3pt);
         \end{tikzpicture}&\begin{tikzpicture}[scale=0.7]
         \path [draw, dotted] (-1,0)--(1,0) (0,-1)--(0,1);
             \path [draw, line width=2pt] (-1,0)--(-0.2,0) (0,0.2)--(0,1);
             \filldraw[red] (0,0) circle (3pt);
         \end{tikzpicture} \\
         \hline
         \Small$w_{u,-q^{-1/2}}(i_1,i_2;j_1,j_2)$& $\frac{1+q^{1/2}u}{1+q^{-1/2}u}$&$\frac{u+q^{-1/2}}{1+q^{-1/2}u}$ &$\frac{u+q^{1/2}}{1+q^{-1/2}u}$&$\frac{(1-q^{-1})u}{1+q^{-1/2}u}$&$\frac{1-q}{1+q^{-1/2}u}$\\\hline \Small$w_{x,-q^{-1/2}}(i_1,i_2;j_1,j_2)$& $\frac{vx}{1-q^{-1}}$&$x$ &$\frac{v}{1-q^{-1}}$&$\frac{-q^{-1}v+x(1-q^{-1})}{1-q^{-1}}$&$\frac{1-q+v x}{1-q^{-1}}$
    \end{tabular}
    \caption{All possible configurations and their weights in the case $s_j=-q^{-1/2}$.}
    \label{fig:weights if s is minus root of q}
\end{figure}

The combinatorial model for the right-hand side of \eqref{eq: relation sHL and mod Robbins} is thus up to a normalization factor given by ensembles of non-crossing paths that must not intersect along horizontal and vertical 
edges in the the higher spin six vertex model, where at each vertex $(j,i)$, for $(j,i)\in \mathbb{Z}_{\ge 0}\times \{1,\dots,n\}$, the vertex weight is given by $w_{x_i,-q^{-1/2}}$ as in the last row of Figure \ref{fig:weights if s is minus root of q}.



The bijection between such ensembles and monotone triangles is as follows: the $x$-coordinate of the up-step going from the 
horizontal line at height $i$ to the horizontal line at height $i+1$ of the $j$-th path, counted from left to right, is the $j$-th entry in row $i$ of the monotone triangle. See Figure \ref{fig:example underlying bijection}, for an example where $\lambda=(8,7,4,3,1,0)$. 
However, as such, the bijection is not yet weight-preserving, as we still need to take the normalization factor 
$$
(1-q^{-1})^{\binom{n}{2}}\prod_{i=1}^n\frac{1-q^{-1}}{1-q+vx_i} 
$$
on the right-hand side of \eqref{eq: relation sHL and mod Robbins} into account.

\begin{figure}[h]
    \centering
    \begin{tikzpicture}
        \draw[step=1cm,gray,very thin] (0,1) grid (8,6);
        \path [draw, line width=2pt, rounded corners] (-1,1)--(2,1)--(2,1.8)--(2.2,2)--(6,2)--(6,3)--(7,3)--(7,4)--(8,4)--(8,7) (-1,2)--(1.8,2)--(2,2.2)--(2,2.8)--(2.2,3)--(4,3)--(4,3.8)--(4.2,4)--(6,4)--(6,5)--(7,5)--(7,7) (-1,3)--(1.8,3)--(2,3.2)--(2,3.8)--(2.2,4)--(3.8,4)--(4,4.2)--(4,7) (-1,4)--(1.8,4)--(2,4.2)--(2,5)--(3,5)--(3,7) (-1,5)--(1,5)--(1,7) (-1,6)--(0,6)--(0,7);
        \node at (0.3,6.5) {\Small$\lambda_6$};
        \node at (1.3,6.5) {\Small$\lambda_5$};
        \node at (3.3,6.5) {\Small$\lambda_4$};
        \node at (4.3,6.5) {\Small$\lambda_3$};
        \node at (7.3,6.5) {\Small$\lambda_2$};
        \node at (8.3,6.5) {\Small$\lambda_1$};

        \node at (1.3,5.5) {\Small$a_{1,5}$};
        \node at (3.3,5.5) {\Small$a_{2,5}$};
        \node at (4.3,5.5) {\Small$a_{3,5}$};
        \node at (7.3,5.5) {\Small$a_{4,5}$};
        \node at (8.3,5.5) {\Small$a_{5,5}$};

        \node at (2.3,4.5) {\Small$a_{1,4}$};
        \node at (4.3,4.5) {\Small$a_{2,4}$};
        \node at (6.3,4.5) {\Small$a_{3,4}$};
        \node at (8.3,4.5) {\Small$a_{4,4}$};

        \node at (2.3,3.5) {\Small$a_{1,3}$};
        \node at (4.3,3.5) {\Small$a_{2,3}$};
        \node at (7.3,3.5) {\Small$a_{3,3}$};

        \node at (2.3,2.5) {\Small$a_{1,2}$};
        \node at (6.3,2.5) {\Small$a_{2,2}$};

        \node at (2.3,1.5) {\Small$a_{1,1}$};

        \node at (-1.5,1) {$1$};
        \node at (-1.5,2) {$2$};
        \node at (-1.5,3) {$3$};
        \node at (-1.5,4) {$4$};
        \node at (-1.5,5) {$5$};
        \node at (-1.5,6) {$6$};

        \node at (1,0.5) {$1$};
        \node at (2,0.5) {$2$};
        \node at (3,0.5) {$3$};
        \node at (4,0.5) {$4$};
        \node at (5,0.5) {$5$};
        \node at (6,0.5) {$6$};
        \node at (7,0.5) {$7$};
        \node at (8,0.5) {$8$};
        \node at (0,0.5) {$0$};
    \end{tikzpicture}\\
    \begin{tikzpicture}[scale=0.5]
    \node at (-1,-1) {$0$};
    \node at (1,-1) {$1$}; 
    \node at (3,-1) {$3$};
    \node at (5,-1) {$4$};
    \node at (7,-1) {$7$};
    \node at (9,-1) {$8$};
    
    \node at (0,0) {$1$};
    \node at (2,0) {$3$}; 
    \node at (4,0) {$4$};
    \node at (6,0) {$7$};
    \node at (8,0) {$8$};
    
    \node at (1,1) {$2$};
    \node at (3,1) {$4$}; 
    \node at (5,1) {$6$};
    \node at (7,1) {$8$};
    
    \node at (2,2) {$2$};
    \node at (4,2) {$4$};
    \node at (6,2) {$7$};
    
    \node at (3,3) {$2$};
    \node at (5,3) {$6$};
    
    \node at (4,4) {$2$};
    \end{tikzpicture}
    \caption{An ensemble of paths and the corresponding monotone triangle}
    \label{fig:example underlying bijection}
\end{figure}
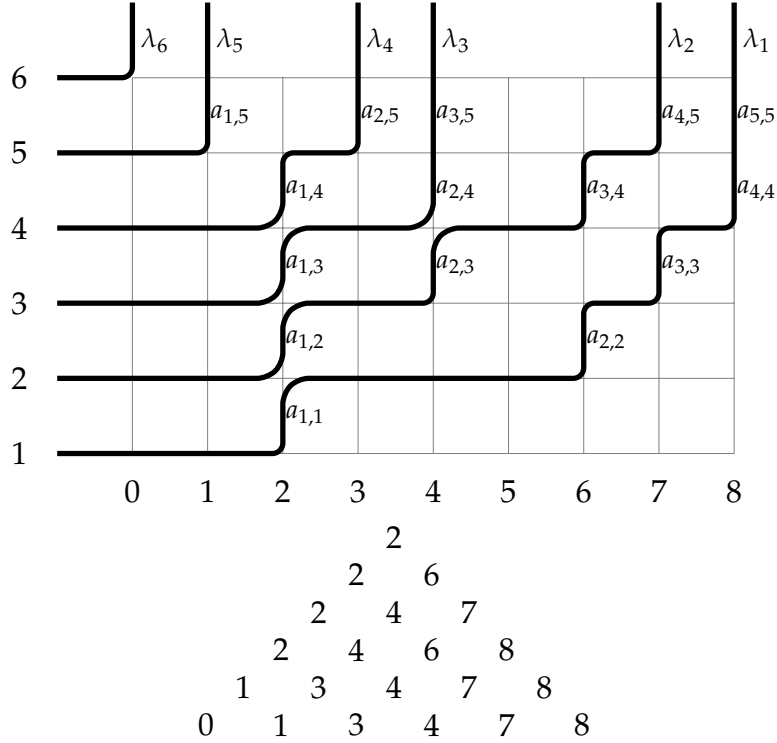

This normalization factor can be interpreted as follows. Note that there is at least one configuration of the form $(0,1;1,0)$ per row. This is because there is one more path exiting each row through a vertical step than entering through a vertical step, due to the left boundary condition. Therefore, not all exits can be of the form $(1,1;0,0)$ or $(1,1;1,1)$ and so at least one has to be of the form $(0,1;1,0)$. The product $\prod_{i=1}^n\frac{1-q^{-1}}{1-q+vx_i}$ can thus be interpreted as cancelling one of the configurations $(0,1;1,0)$ per row---say, without loss of generality, the leftmost such configuration. Further, since there are $i$ paths entering row $i+1$, $i=1,2,\ldots,n$, through a vertical step, there are $\binom{n}{2}$ configurations with $i_1=1$ overall. Their denominators of the weight are taken into account by the prefactor $(1-q^{-1})^{\binom{n}{2}}$. In summary, we can interpret the right-hand side as the generating function of the 
ensemble of paths with the following normalized weights as provided in 
Figure~\ref{fig:weights if s is minus root of q normalized}, with the exception that the 
leftmost configuration of type $(0,1;1,0)$ has weight $1$.
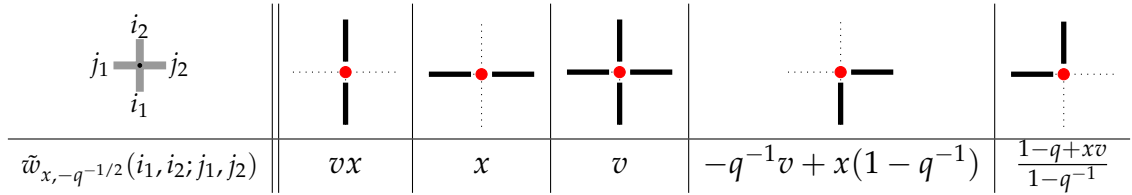
\begin{figure}[h]
    \centering
    \begin{tabular}{c||c|c|c|c|c}
    \begin{tikzpicture}[scale=0.35]
             \path [draw=black!40, line width=3pt] (-1,0)--(1,0) (0,-1)--(0,1);
             \filldraw[black] (0,0) circle (2pt);
             \node at (0,-1.5) {\Small$i_1$};
             \node at (0,1.5) {\Small$i_2$};
             \node at (-1.5,0) {\Small$j_1$};
             \node at (1.5,0) {\Small$j_2$};
         \end{tikzpicture}&
         \begin{tikzpicture}[scale=0.7]
         \path [draw, dotted] (-1,0)--(1,0) (0,-1)--(0,1);
             \path [draw, line width=2pt] (0,-1)--(0,-0.2) (0,0.2)--(0,1);
             \filldraw[red] (0,0) circle (3pt);
         \end{tikzpicture}& \begin{tikzpicture}[scale=0.7]
         \path [draw, dotted] (-1,0)--(1,0) (0,-1)--(0,1);
             \path [draw, line width=2pt] (-1,0)--(-0.2,0) (0.2,0)--(1,0);
             \filldraw[red] (0,0) circle (3pt);
         \end{tikzpicture}&\begin{tikzpicture}[scale=0.7]
         \path [draw, dotted] (-1,0)--(1,0) (0,-1)--(0,1);
             \path [draw, line width=2pt] (0,-1)--(0,-0.2) (0,0.2)--(0,1) (-1,0)--(-0.2,0) (0.2,0)--(1,0);
             \filldraw[red] (0,0) circle (3pt);
         \end{tikzpicture}&\begin{tikzpicture}[scale=0.7]
         \path [draw, dotted] (-1,0)--(1,0) (0,-1)--(0,1);
             \path [draw, line width=2pt] (0,-1)--(0,-0.2) (0.2,0)--(1,0);
             \filldraw[red] (0,0) circle (3pt);
         \end{tikzpicture}&\begin{tikzpicture}[scale=0.7]
         \path [draw, dotted] (-1,0)--(1,0) (0,-1)--(0,1);
             \path [draw, line width=2pt] (-1,0)--(-0.2,0) (0,0.2)--(0,1);
             \filldraw[red] (0,0) circle (3pt);
         \end{tikzpicture} \\
         \hline
         \Small$\tilde{w}_{x,-q^{-1/2}}(i_1,i_2;j_1,j_2)$& $vx$&$x$ &$v$&$-q^{-1}v+x(1-q^{-1})$&$\frac{1-q+xv}{1-q^{-1}}$
    \end{tabular}
    \caption{The normalized weights $\tilde{w}_{x,-q^{-1/2}}$}
    \label{fig:weights if s is minus root of q normalized}
\end{figure}


With these modified weights, we can argue that the bijection is weight-preserving between 
the ensemble of paths and monotone triangles with bottom row $\lambda_n,\dots,\lambda_1$, if we equip this set with the polynomial weight introduced in \eqref{eq: polynomial weight of MTs}. 

Indeed, the configurations of the form $(1,1;0,0)$ in row $i$ correspond to right-leaning entries in row $i-1$ and the configurations of the form $(1,1;1,1)$ in row $i$ correspond to left-leaning entries in row $i-1$, and it can be checked that the weights are chosen appropriately. Moreover, configurations of the form $(1,0;0,1)$ and $(0,1;1,0)$ always appear in pairs in each row, with the exception of the leftmost configuration of the form $(0,1;1,0)$. This is because for any path that enters a horizontal line via a configuration of the form $(1,0;0,1)$, there is another path (in some cases the same) that leaves the same horizontal line via $(0,1;1,0)$ to the right. Indeed, consider a configuration of the form $(1,0;0,1)$, the only configurations that can come after it when traversing the horizontal line from left to right before encountering a configuration of the form $(0,1;1,0)$ are those of the form $(0,0;1,1)$ and of form $(1,1;1,1)$, where in the latter case the path under consideration changes.

The only path not entering via a vertex of the form $(1,0;0,1)$ but entering via the left boundary pairs with the leftmost exit of type $(0,1;1,0)$, which is taken into account by the fact that the leftmost configuration of type $(0,1;1,0)$ has weight $1$. 

Now the configurations of type $(1,0;0,1)$ in row $i$ correspond to special entries in row $i-1$, and their weights multiplied together with the weight of the corresponding configuration of type $(0,1;1,0)$ gives the polynomial weight $x(vx^{-1}+w+vx)$
recalling that $v=u=q^{1/ 2}-q^{-1/ 2}$ and $w=q^{-1}-q$, as desired.  

\begin{rmk}
Note that it is also possible to turn this into a weight-preserving bijection between DAMTs and the ensemble of paths by slightly modifying the paths once more and coloring them as follows. We use the observation from above about the pairing $(1,0;0,1)$'s and $(0,1;1,0)$'s. Now we replace the weights of the configurations of type $(0,1;1,0)$ by $1$ and color the configurations of type $(1,0;0,1)$ with red, blue or yellow, with the corresponding weights being $v$, $wx$ and $vx^2$, respectively. 
\end{rmk} 


\section{Recurrences for $\F_{\lambda}(u_1,\ldots,u_n | s_0,s_1,\ldots)$ and the left-hand sides of the identities}\label{Sec: recurrence relations}

In this section, we first derive a recurrence for $\F_{\lambda}(u_1,\ldots,u_n | s_0,s_1,\ldots)$ with respect to $n$ in Proposition~\ref{rec}. The recurrence will then be used to establish recurrences for 
the left-hand sides of the identities from Theorems~\ref{main1} and \ref{main2} as well as Corollary~\ref{main2:cor}, stated in Corollaries~\ref{rec1}, \ref{rec2}, and \ref{rec2v}, respectively.  These recurrences will be instrumental in Sections \ref{sec: proof of Thm main1} and \ref{sec: proof of Thm main2} to prove Theorems~\ref{main1} and \ref{main2}. 

We use the standard notation $[n] \overset{\textrm{def}}{=} \{1,\ldots,n\}$. Moreover, for a set $T$ and a non-negative integer $k$, denote by  $\left( T \atop k \right)$  the set of $k$-subsets of $T$, and, for a sequence $(u_i)_{1 \le i \le n}$, we let $u_T \overset{\textrm{def}}{=} (u_i)_{i \in T}$ 
if $T$ is a subset of $[n]$.

\begin{prop}
\label{rec}
Let $n \ge 1$, $\lambda$ be a partition of length $n$ allowing zero parts and $k \in \{0,\ldots,n-1\}$ be minimal such that $\lambda_n=\lambda_{n-1}=\ldots,\lambda_{k+1}$. Then 
\begin{multline}\label{eq: recurrence for F lambda}
\F_{\lambda}(u_1,\ldots,u_n | s_0,s_1,\ldots) = (q;q)_{n-k} \prod_{i=1}^n \left( \frac{1}{1-s_l u_i} \prod_{j=0}^{l-1} \frac{u_i-s_j}{1-s_j u_i} \right) \\\times  \sum_{T \in \left( [n] \atop k \right)} \prod_{i \in T} (u_i-s_l) \prod_{i \in T, j \in T^c} \frac{u_i- q u_j}{u_i - u_j} \F_{\lambda^k-l-1}(u_T | s_{l+1},s_{l+2},\ldots), 
\end{multline} 
where $l \overset{\textrm{def}}{=} \lambda_n$, $\lambda^k-l-1 \overset{\textrm{def}}{=} (\lambda_1-l-1,\ldots,\lambda_k-l-1)$ and $T^c \overset{\textrm{def}}{=} [n] \setminus T$. 
\end{prop}

\begin{proof}
Let  
$$
f_\lambda(u_1,\dots,u_n|s_0,s_1,\ldots) \overset{\textrm{def}}{=} \prod_{1\le i<j\le n}\frac{u_i-qu_j}{u_i-u_j}\prod_{i=1}^n\left(\frac{1-q}{1-s_{\lambda_i}u_i}\prod_{j=0}^{\lambda_i-1}\frac{u_i-s_j}{1-s_ju_i}\right), 
$$
be the argument in the symmetrizer in the definition of $\F_\lambda(u_1,\dots,u_n|s_0,s_1,\ldots)$.
It follows that 
\begin{multline*}
f_\lambda(u_1,\dots,u_n|s_0,s_1,\ldots) = 
    (1-q)^{n-k} \prod_{i=1}^k \left( (u_i-s_l) \prod_{j=k+1}^n\frac{u_i-qu_j}{u_i-u_j} \right) \prod_{k+1\le i<j\le n}\frac{u_i-qu_j}{u_i-u_j}\\\times{\prod_{i=1}^n \left(\frac{1}{1-s_lu_i}\prod_{j=0}^{l-1}\frac{u_i-s_j}{1-s_ju_i}\right)}f_{\lambda^k-l-1}(u_1,\dots,u_k|s_{l+1},s_{l+2},\dots).
\end{multline*}
This can be turned into the following recurrence for $\F_{\lambda}(u_1,\ldots,u_n|s_0,s_1,\ldots)$ by computing the symmetrizer by first permuting $u_1,\ldots,u_k$ and $u_{k+1},\ldots,u_n$, separately, and then shuffling $u_1,\ldots,u_k$ and $u_{k+1},\ldots,u_n$ in all possible ways. \begin{multline*} 
\F_{\lambda}(u_1,\ldots,u_n | s_0,s_1,\ldots) = \sum_{T \in \left( [n] \atop k \right)} (1-q)^{n-k} 
\prod_{i \in T, j \in T^c} \frac{u_i- q u_j}{u_i - u_j}
\sym{u_{T^c}} \left[  \prod_{i,j \in T^c\atop i<j} \frac{u_i - q u_j}{u_i-u_j} \right] \\ 
\times \prod_{i \in T} (u_i-s_l) 
\prod_{i=1}^n \left( \frac{1}{1-s_l u_i} \prod_{j=0}^{l-1} \frac{u_i-s_j}{1-s_j u_i} \right)
\F_{\lambda^k-l-1}(u_T | s_{l+1},s_{l+2},\ldots)
\end{multline*} 
The result then follows from the following identity for
$m \ge 0$  
$$
\Symm{u}{m} \left[ \prod_{1 \le i < j \le m} \frac{u_i - q u_j}{u_i-u_j} \right] = \frac{(q;q)_m}{(1-q)^m}, 
$$
which can for instance be found in \cite[p. 207, (1.4)]{Macdonald}.
\end{proof}

Proposition~\ref{rec} can be used in a straightforward manner to derive a recurrence for the left-hand side of the identity in Theorem~\ref{main1}. We denote this left-hand side by 
$$
H_1(u_1,\ldots,u_n|s_0,s_1,\ldots) \overset{\textrm{def}}{=} 
\sum_{\lambda_1 \ge \lambda_2 \ge \ldots \ge \lambda_n \ge 0} 
\prod_{r \ge 0} \frac{(-s_r;q)_{m_r(\lambda)}}{(q;q)_{m_r(\lambda)}} \F_{\lambda}(u_1,\ldots,u_n|s_0,s_1,\ldots).
$$

\begin{cor} 
\label{rec1} 
For any $n \ge 1$, we have 
\begin{multline*} 
H_1(u_1,\ldots,u_n|s_0,s_1,\ldots) = \sum_{l \ge 0} \sum_{T \subsetneq [n]} (-s_l;q)_{n-|T|} 
\prod_{i \in T} (u_i-s_l) \\ \times  \prod_{i=1}^n \left( \frac{1}{1-s_l u_i} \prod_{j=0}^{l-1} \frac{u_i-s_j}{1-s_j u_i} \right) 
\prod_{i \in T,j \in T^c} \frac{u_i-q u_j}{u_i-u_j} H_1(u_T|s_{l+1},s_{l+2},\ldots).
\end{multline*} 
\end{cor} 
\begin{proof}
    Let us denote by $n-k$ the multiplicity of the smallest part $l=\lambda_n$ of $\lambda$. Then we can rewrite $H_1(u_1,\ldots,u_n|s_0,s_1,\ldots)$ as
\begin{equation*}
    \sum_{l\ge 0}\sum_{k=0}^{n-1}\frac{(-s_l;q)_{n-k}}{(q;q)_{n-k}}\sum_{\lambda_1 \ge \lambda_2 \ge \ldots \ge \lambda_k > l} 
\prod_{r > l} \frac{(-s_r;q)_{m_r(\lambda)}}{(q;q)_{m_r(\lambda)}} \F_{(\lambda_1,\dots,\lambda_k,l,\dots,l)}(u_1,\ldots,u_n|s_0,s_1,\dots).
\end{equation*}
Using \eqref{eq: recurrence for F lambda}, this becomes
\begin{multline*}
    \sum_{l\ge 0}\sum_{k=0}^{n-1}\frac{(-s_l;q)_{n-k}}{(q;q)_{n-k}}\sum_{\lambda_1 \ge \lambda_2 \ge \ldots \ge \lambda_k > l} 
    \prod_{r>l} \frac{(-s_r;q)_{m_r(\lambda)}}{(q;q)_{m_r(\lambda)}} (q;q)_{n-k}\left(\prod_{i=1}^n\frac{1}{1-s_lu_i}\prod_{j=0}^{l-1}\frac{u_i-s_j}{1-s_ju_i}\right)\\\times\sum_{T\in\left([n]\atop k\right)}\prod_{i\in T}(u_{i}-s_l) \prod_{i\in T, j\in T^c}\frac{u_{i}-qu_{j}}{u_{i}-u_{j}}  F_{\lambda^k-l-1}(u_T|s_{l+1},s_{l+2},\dots).
\end{multline*}
Now, interchanging the two inner sums and regrouping some terms yields the claim.
\end{proof}

Next, we consider the left-hand side of the identity in Theorem~\ref{main2}, that is 
\begin{multline*} 
H_2(u_1,\ldots,u_n|s_0,s_1,\ldots) \overset{\textrm{def}}{=}
\sum_{\lambda_1 \ge \lambda_2 \ge \ldots \ge \lambda_n \ge 0} 
\frac{(-\gamma q^{1/ 2};q^{1 / 2})_{m_0(\lambda)}}{(q;q)_{m_0(\lambda)}} (-\gamma^{-1}s_0;q^{1 / 2})_{m_0(\lambda)} \\
\times
\prod_{r \ge 1} \frac{(-s_r;q^{1 /2})_{m_r(\lambda)}}{(q^{1 /2};q^{1 /2})_{m_r(\lambda)}}  \F_{\lambda}(u_1,\ldots,u_n|s_0,s_1,\ldots).
\end{multline*} 
and denote its specialization at $\gamma=1$ by $\widehat{H}_2(u_1,\ldots,u_n|s_0,s_1,\ldots)$, which is the left-hand side of 
Corollary~\ref{main2:cor}.

Again, Proposition~\ref{rec} implies the following recurrence by taking into account the multiplicities of the smallest part $\lambda_n$, and interchanging sums. It is necessary to distinguish between the cases 
$l=\lambda_n=0$ and $l=\lambda_n >0$. Note that Iverson notation is used to express this, i.e. we have for any statement $[\textrm{statement}]=1$ if the statement is true and $[\textrm{statement}]=0$ otherwise.

\begin{cor} 
\label{rec2} 
For any $n \ge 1$, we have 
\begin{multline*} 
H_2(u_1,\ldots,u_n|s_0,s_1,\ldots) = \sum_{l \ge 0} \sum_{T \subsetneq [n]} 
(-\gamma^{[l=0]} q^{1/ 2};q^{1 / 2})_{n-|T|} (-\gamma^{-[l=0]}s_l;q^{1 / 2})_{n-|T|} \\
\times \prod_{i \in T} (u_i-s_l)  \prod_{i=1}^n \left( \frac{1}{1-s_l u_i} \prod_{j=0}^{l-1} \frac{u_i-s_j}{1-s_j u_i} \right) 
\prod_{i \in T,j \in T^c} \frac{u_i-q u_j}{u_i-u_j} \widehat{H}_2(u_T| s_{l+1},s_{l+2},\ldots).
\end{multline*} 
\end{cor}

As a corollary of Corollary~\ref{rec2}, we obtain the following by setting $\gamma=1$. Corollary~\ref{rec2v} is actually needed before 
Corollary~\ref{rec2} as we first derive Corollary~\ref{main2:cor} using Corollary~\ref{rec2v} and then deduce Theorem~\ref{main2} 
from Corollary~\ref{main2:cor} using Corollary~\ref{rec2}, see Section \ref{sec: proof of Thm main2}.

\begin{cor} 
\label{rec2v} 
For any $n \ge 1$, we have 
\begin{multline*} 
\widehat{H}_2(u_1,\ldots,u_n|s_0,s_1,\ldots) = \sum_{l \ge 0} \sum_{T \subsetneq [n]} 
(-q^{1/ 2};q^{1 /2})_{n-|T|} (-s_l;q^{1 /2})_{n-|T|} \\ \times
\prod_{i \in T} (u_i-s_l) \prod_{i=1}^n \left( \frac{1}{1-s_l u_i} \prod_{j=0}^{l-1} \frac{u_i-s_j}{1-s_j u_i} \right) 
\prod_{i \in T, j \in T^c} \frac{u_i-q u_j}{u_i-u_j} \widehat{H}_2(u_T|s_{l+1},s_{l+2},\ldots).
\end{multline*} 
\end{cor}

\section{Proof of Theorem \ref{main1}}\label{sec: proof of Thm main1}

We prove Theorem~\ref{main1} by induction with respect to $n$. The case $n=0$ is obvious. Regarding the induction step, we use Corollary~\ref{rec1} to show that the right-hand side of the identity in Theorem~\ref{main1}  satisfies the recursion that we established for the left-hand side in Corollary~\ref{rec1}. 
Concretely, we need to prove 
\begin{multline}
\prod_{i=1}^n \frac{1}{1-u_i} \prod_{1 \le i < j \le n} \frac{1-q u_i u_j}{1-u_i u_j}  = \sum_{l \ge 0} \sum_{T \subsetneq [n]} (-s_l;q)_{n-|T|} 
\prod_{i \in T} (u_i-s_l) \\ \times  \prod_{i=1}^n \left( \frac{1}{1-s_l u_i} \prod_{j=0}^{l-1} \frac{u_i-s_j}{1-s_j u_i} \right) 
\prod_{i \in T, j \in T^c} \frac{u_i-q u_j}{u_i-u_j} \prod_{i \in T} \frac{1}{1-u_i} \prod_{ i,j \in T\atop i<j} \frac{1-q u_i u_j}{1-u_i u_j}.
\label{eq:A}
\end{multline}

The key ingredient will be the polynomial identity stated in the following lemma. For a subset $T$ of another set $S$, we define 
\begin{equation}\label{eq: def of sign of set}
    \sgn_S(T)  \overset{\textrm{def}}{=} (-1)^{\# \{i \in T, j \in S \setminus T | i>j\}}.
\end{equation}

\begin{lem}\label{lem: Key lemma for main1}
For any $n \ge 1$ and a variable $s$, we have 
\begin{multline*} 
\prod_{i=1}^n (1- s u_i) \prod_{1 \le i < j \le n} (1- q u_i u_j)(u_i-u_j) 
= \sum_{T \subseteq [n]} \sgn_{[n]}(T)  (-s;q)_{n-|T|} \prod_{i \in T^c} (1-u_i)   
\\ \times \prod_{i \in T} (u_i - s)  \prod_{ i,j \in T^c\atop i<j} (1-u_i u_j)(u_i-u_j)
\prod_{i \in T,j \in T^c} (u_i-q u_j)(1- u_i u_j)  \prod_{i,j \in T\atop i<j} (1 - q u_i u_j)(u_i-u_j).
\end{multline*} 
\end{lem} 
\begin{proof} The proof is by induction with respect to $n$. The identity is obvious for $n=0$.  
Both sides of the identity are polynomials in $u_n$ of degree no greater than $2n-1$. We provide $2n$ evaluations on which they agree. 
More concretely, these are 
\begin{itemize}
\item $u_n=u_k$ for $1 \le k \le n-1$, 
\item $u_n=u_k^{-1}$ for $1 \le k \le n-1$, 
\item $u_n=1$, 
\item $u_n=s$.
\end{itemize}

\emph{Case $u_n=u_k$, $1 \le k \le n-1$:}  The left-hand side of the identity obviously vanishes, and it remains to show that this is also the 
case for the right-hand side. 

A summand of the right-hand side vanishes if either $k, n \in T$ or $k, n \in T^c$. Therefore, we can assume that $k \in T$ and $n \in T^c$, or 
$k \in T^c$ and $n \in T$. If we are considering a $T$ of the first type, then $T'=T \setminus \{k\} \cupdot \{n\}$ is of the second type. Note that here $A\cupdot B$ denotes the disjoint union of two sets $A$ and $B$. It can be checked that 
the summands of $T$ and $T'$ agree under this specialization up to sign, which establishes a sign-reversing involution of the summands under this specialization. 

\emph{Case $u_n=u_k^{-1}$, $1 \le k \le n-1$:} The left-hand side of the identity becomes 
\begin{multline} 
\label{LHSinv}
(-1)^{k+1} (1- s u_k)(1- s u_k^{-1}) \prod_{1 \le i \le n-1\atop i \not= k} (1-s u_i)(u_k-u_i)(1- q u_i u_k) \\ 
\times \prod_{i=1}^{n-1} (1-q u_i u_k^{-1})(u_i-u_k^{-1}) \prod_{1 \le i < j \le n-1\atop i,j \not= k} (1-q u_i u_j)(u_j-u_i).
\end{multline} 

Concerning the right-hand side, note that the appearance of the factor $1- u_i u_j$ in the product over all $i<j$ with $i,j \in T^c$ and over all $i,j$ with 
$i \in T, j \in T^c$ implies that the summand vanishes unless $k,n \in T$. Writing every such $T$ as $T=U\cupdot\{k,n\}$ and after some manipulations, 
we see that the right-hand side is under this specialization equal to
\begin{multline*} 
(-1)^{k+1} (u_k^{-1}-s)(u_k-s) \prod_{i=1}^{n-1} (1- q u_i u_k^{-1})(u_i-u_k^{-1}) 
\prod_{1 \le i \le n-1\atop i \not=k} (u_k-u_i)(1- q u_i u_k)  \\
\times \sum_{U \subseteq [n-1] \setminus \{k\}} \sgn_{[n-1] \setminus \{k\}} (U) (-s;q)_{n-2-|U|}
\prod_{i \in U^c} (1-u_i) \prod_{i \in U} (u_i-s) \\
\times \prod_{i,j \in U^c\atop i<j} (1-u_i u_j)(u_i-u_j) 
\prod_{i \in U, j \in U^c} (u_i-q u_j)(1-u_i u_j) 
\prod_{i,j \in U\atop i<j} (1-q u_i u_j)(u_i-u_j),
\end{multline*} 
where $U^c$ is the complement of $U$ in $[n-1] \setminus \{k\}$. By induction, this is equal to 
\begin{multline*}
(-1)^{k+1} (u_k^{-1}-s)(u_k-s) \prod_{i=1}^{n-1} (1- q u_i u_k^{-1})(u_i-u_k^{-1}) 
\prod_{1 \le i \le n-1\atop i \not=k} (u_k-u_i)(1- q u_i u_k)  \\ 
\times \prod_{1 \le i \le n-1\atop i \not=k} (1- s u_i) \prod_{1 \le i < j \le n-1\atop i,j \not= k} (1- q u_i u_j)(u_i-u_j), 
\end{multline*} 
which is a simple rewriting of \eqref{LHSinv}.

\emph{Case $u_n=1$:} The left-hand side simplifies to 
\begin{equation}
\label{LHS1} 
(1-s) \prod_{i=1}^{n-1} (1- q u_i)(u_i-1) \prod_{i=1}^{n-1} (1- s u_i) \prod_{1 \le i < j \le n-1} (1- q u_i u_j)(u_i-u_j).
\end{equation} 
As for the right-hand side, only summands accompanying subsets $T \subseteq [n]$ that contain $n$ are non-vanishing. Writing every such $T$ as $T=U\cupdot\{n\}$, the right-hand side is equal to 
\begin{multline*}
(1-s) \prod_{i=1}^{n-1} (1- q  u_i)(u_i-1) \sum_{U \subseteq [n-1]} 
\sgn_{[n-1]}(U) 
(-s;q)_{n-1-|U|} \prod_{i \in U^c} (1-u_i) \\
\times\prod_{i \in U} (u_i-s)   \prod_{i,j \in U^c\atop i<j} (1-u_i u_j)(u_i-u_j) 
\prod_{i \in U, j \in U^c} (u_i- q u_j)(1- u_i u_j) 
\prod_{i,j \in U\atop i<j} (1- q u_i u_j)(u_i - u_j), 
\end{multline*}
where $U^c$ is the complement of $U$ in $[n-1]$ and this is obviously equal to \eqref{LHS1} by induction. 

\emph{Case $u_n=s$:} Under this specialization, the left-hand side is equal to 
$$
(1-s^2) \prod_{i=1}^{n-1} (1-s u_i)(1- q s u_i)(u_i-s) \prod_{1 \le i < j \le n-1} (1-q u_i u_j)(u_i-u_j).
$$
Concerning the right-hand side, we may assume $n \in T^c$ in this case, as otherwise the corresponding summand vanishes. The right-hand side can now be written as 
\begin{multline*} 
(1-s) \prod_{i=1}^{n-1} (1 - s u_i)(u_i-s) \sum_{T \subseteq [n-1]} \sgn_{[n-1]}(T) (-s;q)_{n-|T|}  
\prod_{i \in T^c} (1-u_i)  \prod_{i \in T} (u_i- q s)  \\
\times \prod_{i,j \in T^c\atop i<j} (1-u_i u_j)(u_i-u_j) 
\prod_{i \in T, j \in T^c} (u_i-q u_j)(1-u_i u_j)  
\prod_{i,j \in T\atop i<j} (1- q u_i u_j)(u_i-u_j), 
\end{multline*}
where $T^c$ is the complement of $T$ in $[n-1]$. We use 
$$
(1-s) (-s;q)_{n-|T|}  = (1-s^2) (-s q;q)_{n-1-|T|},
$$
and then the induction hypothesis for $n-1$, where we replace $s$ by  $s q$. 
\end{proof} 

We now reduce \eqref{eq:A} to Lemma \ref{lem: Key lemma for main1}, to finish the proof of Theorem \ref{main1}.
\begin{proof}[Proof of Theorem \ref{main1}]
    We need to show \eqref{eq:A} for any fixed $p \ge 0$ where we set $s_l=s_p$ for $l \ge p$. We split the sum over $l$ on the right-hand side into the part where $l < p$ and the remainder where $l \ge p$. The sum over $l \ge p$ can then be computed using the formula for the geometric series. After further multiplying both sides by $1-\prod_{i=1}^n \frac{u_i-s_p}{1- s_p u_i}$, we see that \eqref{eq:A} is equivalent to the following identity.
\begin{multline}\label{A''}
    \left( 1-\prod_{i=1}^n \frac{u_i-s_p}{1- s_p u_i} \right)\prod_{i=1}^n \frac{1}{1-u_i} \prod_{1 \le i < j \le n} \frac{1-q u_i u_j}{1-u_i u_j} = \sum_{T \subsetneq [n]} \prod_{i \in T,j \in T^c} \frac{u_i-q u_j}{u_i-u_j} \prod_{i \in T} \frac{1}{1-u_i}\\
    \times \prod_{i,j \in T\atop i<j} \frac{1-q u_i u_j}{1-u_i u_j} \left( 
    \sum_{l=0}^{p} (-s_l;q)_{n-|T|} \prod_{i \in T} (u_i-s_l) 
    \prod_{i=1}^n \left( \frac{1}{1-s_l u_i} \prod_{j=0}^{l-1} \frac{u_i-s_j}{1-s_j u_i} \right) \right. \\
    \left. -\prod_{i=1}^n \frac{u_i-s_p}{1- s_p u_i} \sum_{l=0}^{p-1} (-s_l;q)_{n-|T|} \prod_{i \in T} (u_i-s_l) 
    \prod_{i=1}^n \left( \frac{1}{1-s_l u_i} \prod_{j=0}^{l-1} \frac{u_i-s_j}{1-s_j u_i} \right) \right)
\end{multline}
Next we show that \eqref{A''} follows from 
\begin{multline}\label{a}
    \prod_{i=1}^n \prod_{j=0}^{l-1} \frac{u_i-s_j}{1-s_j u_i} 
    \left( 1-\prod_{i=1}^n \frac{u_i-s_l}{1- s_l u_i} \right) 
    \prod_{i=1}^n \frac{1}{1-u_i} \prod_{1 \le i < j \le n} \frac{1-q u_i u_j}{1-u_i u_j} \\
    = \sum_{T \subsetneq [n]} \prod_{i \in T, j \in T^c} \frac{u_i-q u_j}{u_i-u_j} \prod_{i \in T} \frac{1}{1-u_i} \prod_{i,j \in T\atop i<j} \frac{1-q u_i u_j}{1-u_i u_j} \\
    \times (-s_l;q)_{n-|T|} \prod_{i \in T} (u_i-s_l) 
    \prod_{i=1}^n \left( \frac{1}{1 - s_l u_i} \prod_{j=0}^{l-1} \frac{u_i-s_j}{1-s_j u_i} \right), 
\end{multline} 
where  $l$ is a non-negative integer. Indeed, we obtain \eqref{A''} after summing \eqref{a} over $l \in \{0,\ldots,p\}$ and then subtracting the sum over $l \in \{0,\ldots,p-1\}$ multiplied by $\prod_{i=1}^n \frac{u_i-s_p}{1-s_p u_i}$. For the right-hand side of \eqref{A''}, this is immediate, while for the left-hand side, this follows by telescoping. Further straightforward manipulations of \eqref{a} and replacing $s_l$ by $s$ now yield the identity from Lemma~\ref{lem: Key lemma for main1}.
\end{proof}

\section{Proof of Theorem \ref{main2}}\label{sec: proof of Thm main2}
The idea of the proof is similar to that of Theorem \ref{main1}. We aim to show that the right-hand side of the identity in Theorem~\ref{main2} satisfies the recursion for the left-hand side which was derived in Corollary~\ref{rec2}. 

However, Corollary \ref{rec2} relates $H_2(u_1,\dots ,u_n| s_0,s_1,\dots)$ to $\widehat{H}_2(u_T| s_l,s_{l+1},\dots)$ and not to $H_2(u_T| s_l,s_{l+1},\dots)$, so we cannot simply use induction as in the case of Theorem \ref{main1}. Noting that the right-hand side of Corollary \ref{main2:cor} provides a Pfaffian expression for $\widehat{H}_2(u_1,\dots,u_n | s_0,s_{1},\dots)$, we will first prove Corollary~\ref{main2:cor} and then insert that Pfaffian expression into the right-hand side of Corollary~\ref{rec2} and give a proof of the corresponding identity \eqref{eq:to show for main 2} to eventually prove Theorem~\ref{main2}. 

The idea of proof for Corollary~\ref{main2:cor} itself is the same as that of Theorem \ref{main1}. We use induction with respect to $n$; the case $n=0$ is again obvious. Regarding the induction step, we use Corollary~\ref{rec2v} this time: we aim to show that the right-hand side of the identity in Corollary~\ref{main2:cor} also satisfies the recursion for the left-hand side stated in Corollary~\ref{rec2v}.

Throughout the remainder of the paper, for a set $T$ of positive integers, we denote $T_0 \overset{\textrm{def}}{=} T \cup \{0\}$.
Thus, in order to show Corollary~\ref{main2:cor} using Corollary~\ref{rec2v} following the approach we have outlined above, we need to show 
\begin{multline}\label{to be eased by notation}
    \prod_{i=1}^n \frac{1+q^{1/2}}{1-u_i}
    \prod_{1 \le i < j \le n} \frac{1-q  u_i u_j}{u_i-u_j} \\ \times
    \pf_{\chi_{\mathrm{even}}(n) \le i < j \le n} \left(
    		\begin{cases}
    			1, & i=0,\\
    			\frac{(u_i-u_j)((1+q)(1- q^{1 /2} u_i u_j) + (u_i + u_j)(q^{1 /2}- q)) }{(1+q^{1/2})(1- u_i u_j)(1- q u_i u_j)},& i\ge 1,
    		\end{cases} \right)\\= \sum_{l\ge 0}\sum_{T\subsetneq[n]}(-q^{1 /2};q^{1 /2})_{n-\vert T\vert } (-s_l;q^{1 /2})_{n-\vert T\vert }\prod_{i\in T}\frac{(1+q^{1/2})(u_{i}-s_l)}{1-u_i}\\\times \prod_{i=1}^n \left(\frac{1}{1-s_lu_i} \prod_{j=0}^{l-1}\frac{u_i-s_j}{1-s_ju_i}\right)\prod_{i\in T,j\in T^c}\frac{u_{i}-qu_{j}}{u_{i}-u_{j}}
    \prod_{i,j\in T\atop i<j} \frac{1-q  u_i u_j}{u_i-u_j}\\\times \pf_{\chi_{\textrm{even}}(\vert T\vert)\le i<j\le n\atop i,j\in T_0} \left(
    		\begin{cases}
    			1, & i=0,\\
    			\frac{(u_i-u_j)((1+q)(1- q^{1 /2} u_i u_j) + (u_i + u_j)(q^{1 /2}- q)) }{(1+q^{1/2})(1- u_i u_j)(1- q u_i u_j)},& i\ge 1,
    		\end{cases} \right).
\end{multline}
After this, to deduce Theorem~\ref{main2} from Corollary~\ref{main2:cor}, we need to show 
\begin{multline}\label{eq:to show for main 2}
    \prod_{i=1}^n \frac{1+q^{1/2}}{1-u_i}
\prod_{1 \le i < j \le n} \frac{1-q  u_i u_j}{u_i-u_j} \\
\times 
\pf_{\chi_{\mathrm{even}}(n) \le i < j \le n} \left(
		\begin{cases}
			1 + (\gamma-1) \frac{(q^{1/2} - \gamma^{-1} s_0)(1-u_j)}{(1+q^{1/2}) (1- s_0 u_j)}, & i=0,\\
			\frac{(u_i-u_j)((1+q)(1- q^{1 /2} u_i u_j) + (u_i + u_j)(q^{1 /2}- q)  +
			(\gamma-1) f(u_i,u_j))}{(1+q^{1/2})(1- u_i u_j)(1- q u_i u_j)},& i\ge 1,
		\end{cases} \right)  \\= \sum_{l \ge 0} \sum_{T \subsetneq [n]} 
(-\gamma^{[l=0]} q^{1/ 2};q^{1 / 2})_{n-|T|} (-\gamma^{-[l=0]}s_l;q^{1 / 2})_{n-|T|} \prod_{i \in T} (u_i-s_l)  \\
\times\prod_{i=1}^n \left( \frac{1}{1-s_l u_i} \prod_{j=0}^{l-1} \frac{u_i-s_j}{1-s_j u_i} \right) 
\prod_{i \in T, j \in T^c} \frac{u_i-q u_j}{u_i-u_j}\prod_{i\in T} \frac{1+q^{1/2}}{1-u_i}
    \prod_{i,j\in T\atop i<j} \frac{1-q  u_i u_j}{u_i-u_j} \\
    \times 
    \pf_{\chi_{\textrm{even}}(\vert T\vert)\le i<j\le n\atop i,j\in T_{0}} \left(
    		\begin{cases}
    			1, & i=0,\\
    			\frac{(u_i-u_j)((1+q)(1- q^{1 /2} u_i u_j) + (u_i + u_j)(q^{1 /2}- q)) }{(1+q^{1/2})(1- u_i u_j)(1- q u_i u_j)},& i\ge 1,
    		\end{cases} \right)
\end{multline}
by Corollary~\ref{rec2}, where the definition of $f(u,v)$ is given in Theorem~\ref{main2}.

In Sections \ref{subsec: Proof of Corollary main2} and \ref{subsec: Proof of Thm main2}, we reduce both of these identities \eqref{to be eased by notation} and \eqref{eq:to show for main 2} to polynomial identities, which in turn will both follow from the identity \eqref{eq: key Lemma 2} stated in Lemma \ref{lem: Key lemma for main2}.

Before we state this lemma, we introduce some notation. We denote, for fixed $n$, by $\Mgamma$ the matrix that is 
obtained by extending the following triangular array
\begin{equation}\label{eq: Def of Mgamma}
    \left(
		\begin{cases}
			1 + (\gamma-1) \frac{(q^{1/2} - \gamma^{-1} s_0)(1-u_j)}{(1+q^{1/2}) (1- s_0 u_j)}, & i=0,\\
			\frac{(u_i-u_j)((1+q)(1- q^{1 /2} u_i u_j) + (u_i + u_j)(q^{1 /2}- q)  +
			(\gamma-1) f(u_i,u_j))}{(1+q^{1/2})(1- u_i u_j)(1- q u_i u_j)},& i\ge 1,
		\end{cases} \right)_{\chi_{\mathrm{even}}(n) \le i < j \le n}
\end{equation}
to a skew-symmetric matrix, 
which is just the matrix underlying the Pfaffian on the left-hand side of equation \eqref{eq:to show for main 2}, and by replacing all $s_0$ by $s$. The specialization at $\gamma=1$ is 
denoted by $M\Def \Mone$, which is the matrix underlying the Pfaffian on the right-hand sides of \eqref{to be eased by notation} and \eqref{eq:to show for main 2}, 
and also on the left-hand side of \eqref{to be eased by notation}. 

Since we aim to reduce the equations \eqref{to be eased by notation} and \eqref{eq:to show for main 2} to polynomial identities in order to 
prove them using elementary polynomial arguments, it will be useful to multiply $\Mgamma$ by simple matrices so that the 
entries become polynomials in $u_1,\dots,u_n$. To do this, we define a diagonal matrix $B(U,V)$ for two subsets $U,V\subseteq[n]$ by
\begin{equation}\label{eq: Def of B(T)}
   B(U,V)\Def  \diag_{\chi_{\textrm{even}}(\vert V\vert)\le i \le n:i \in V_0 }  \left(\begin{cases}  1 & i=0 \\
                \prod\limits_{i  <  k \le n \atop k\in U}(1-u_iu_k)(1-qu_iu_k) & i \ge 1
            \end{cases} \right),
\end{equation}
and note that
\begin{equation}\label{eq:det of B(T)}
    \det\left(B(U,V) \right)=\prod_{i\in V,j\in U \atop i<j}(1-u_iu_j)(1-qu_iu_j).
\end{equation}
We also introduce the following notation. For a matrix $A=(a_{i,j})_{0 \le i,j\le n}$ and a subset $U\subseteq [n]$, we set
\begin{equation}\label{eq: Notation for matrix restricted 1}
    A_U\Def (a_{i,j})_{\chi_{\textrm{even}}(\vert U \vert)\le i,j\le n: i,j\in U_0}.
\end{equation}
Furthermore, we denote the matrix-product conjugation with $B(U,U)$ by
\begin{equation}\label{eq: Def of Conj by B(T)}
    \overline{A_U}\Def B(U,U) \cdot A_U \cdot B(U,U).
\end{equation}
The following observation is one essential ingredient in our inductive proofs of \eqref{to be eased by notation} and \eqref{eq:to show for main 2} later on: for $S\subseteq T\subseteq[n]$, we have 
\begin{equation}\label{eq: relation to translate}
    (\overline{A_T})_S=B(T\setminus S,S)\cdot \overline{A_S}\cdot B(T\setminus S,S).
\end{equation}

Since we are dealing with Pfaffians here, we need to translate \eqref{eq: relation to translate} into a relation for Pfaffians. To this end, we summarize some basic properties of the Pfaffians; they will be used in other proofs as well. Let $A=(a_{i,j})_{1 \le i,j\le n}$ be a skew symmetric matrix, then the Pfaffian of $A$ has the following properties.
\begin{enumerate}
        \item For any matrix $B=(b_{i,j})_{1\le i,j\le n}$, we have 
        \begin{equation}\label{eq: Multiplicativity of Pfaffian}
            \pf(B \cdot A \cdot B^T)=\det(B)\pf(A).
        \end{equation}
		\item The Pfaffian is antisymmetric: for a permutation $\sigma\in \mathfrak{S}_n$, let $A_{\sigma}$ denote the matrix 
		obtained from $A$ by permuting rows and columns simultaneously according to $\sigma$. Then  
        \begin{equation}\label{eq: Antisymmetry of Pfaffian}
		    \pf(A_{\sigma})=\sgn(\sigma)\pf(A).
		\end{equation}
		This implies in particular that if any two rows (or, equivalently, any two columns) of $A$ agree, then the Pfaffian of $A$ is zero.
		Note that \eqref{eq: Antisymmetry of Pfaffian} is a consequence of \eqref{eq: Multiplicativity of Pfaffian} as the simultaneous permutation of rows and columns can be obtained by conjugating by the corresponding permutation matrix.
		\item The equivalent of the Laplace expansion formula for Pfaffians is as follows: for any $p \in [n]$ 
                 \begin{equation}\label{eq: Laplace for Pfaffian}
                    \pf_{1 \le i < j \le n} \left(A \right) = 
                    \sum_{q=1}^{p-1} (-1)^{p+q+1} a_{q,p} \pf_{1 \le i < j \le n \atop i,j \not= p,q} \left( a_{i,j} \right) \\
                    \\+ \sum_{q=p+1}^{n} (-1)^{p+q+1} a_{p,q} \pf_{1 \le i < j \le n \atop i,j \not= p,q} \left( a_{i,j} \right).
                 \end{equation} 
\end{enumerate}
Combining \eqref{eq:det of B(T)}, \eqref{eq: Def of Conj by B(T)} and \eqref{eq: Multiplicativity of Pfaffian} it follows that 
\begin{equation}\label{eq:Pfaffian of conj by B(T)}
    \pf\left(\overline{A_U}\right) =\det\left(B(U,U)\right)\pf(A_U)=\prod_{i,j\in U\atop i<j}(1-u_iu_j)(1-qu_iu_j)\pf(A_U),
\end{equation}
while combining 
\eqref{eq:det of B(T)}, \eqref{eq: relation to translate} and \eqref{eq: Multiplicativity of Pfaffian} it follows that
\begin{equation}\label{eq: Commutation relation}  
    \pf \left(\left(\overline{A_T}\right)_S\right)=\det\left(B(T\setminus S,S)\right)\pf(\overline{A_S})=\prod_{i\in S,j\in T\setminus S\atop i<j}(1-u_iu_j)(1-qu_iu_j) \pf(\overline{A_S}).
\end{equation}
We will make extensive use of these identities in the remainder of the paper. We also continue using the sign of a set relativ to a superset introduced in \eqref{eq: def of sign of set}.

As we will see in Section \ref{subsec: Proof of Corollary main2}, the special case $\gamma=1$ of the following lemma implies \eqref{to be eased by notation}, and hence also Corollary \ref{main2:cor}. In Section \ref{subsec: Proof of Thm main2}, we will then see that the lemma, combined with Corollary \ref{main2:cor}, implies \eqref{eq:to show for main 2} and thus also Theorem \ref{main2}.

\begin{lem}\label{lem: Key lemma for main2} For any $n \ge 1$ and a variable $s$, we have
\begin{multline}\label{eq: Key Lemma main2 identity A}
(1+q^{1/2})(1-s^2)
    \pf\left(\overline{\Mgamma}|_{u_1=s}\right) \prod_{i=2}^n(1-su_i) \\ =(1+\gamma q^{1/2})(1+\gamma^{-1}s)
    (1-s) \pf\left(\overline{\Mgammarootqbracket_{[n]\setminus\{1\}}}|_{s=qs}\right) \prod_{i=2}^n(s-u_i)(1-su_i) (1-squ_i)
\end{multline}
and
    \begin{multline}\label{eq: key Lemma 2}
    \prod_{i=1}^n (1+q^{1/2})(1-su_i)
    \pf\left(\overline{\Mgamma}\right)=  \sum_{T\subseteq[n]} \sgn_{[n]}(T) (-\gamma^{-1} s;q^{1 /2})_{n-\vert T\vert }(-\gamma q^{1 /2};q^{1 /2})_{n-\vert T\vert } \\ \times \prod_{i\in T,j\in T^c}(u_{i}-qu_{j})(1-u_iu_j) \prod_{i\in T^c}(1-u_i)\prod_{i\in T}(1+q^{1/2})(u_{i}-s) \\ 
   \times  \prod_{i,j\in T^c\atop i<j}(1-u_iu_j)(u_i-u_j) \pf\left(\overline{M_T}\right).
\end{multline}     
\end{lem}
\begin{proof}
We start by proving {\eqref{eq: Key Lemma main2 identity A}}.
The proof is by induction with respect to $n$. The case $n=0$ is easy, and we assume $n>0$ from now on. 
As for the induction step, note that both sides are polynomials in $u_i$, $i=2,\ldots,n$, of degree at most $2n-1$, and, therefore, 
it suffices to find $2n$ evaluations for, say, $u_2$ on which both sides agree. 

Note that we do not cancel the common factor $\prod_{i=2}^n (1-su_i) $ on both sides, since this is needed to cancel the denominator in  $f(u_i,u_j)$ in the $(i,j)$-th entry of the two matrices on both sides to remain polynomial. This factor also directly produces the first simple evaluation $u_2=s^{-1}$. We provide $3n-6$ more evaluations on which both sides agree. For $n \in \{0,1,\ldots,4\}$, this equation can easily be checked by a computer, while for $n\ge 5$, we have $3n-6+1\ge 2n$ and so we found enough evaluations. More concretely, these are
\begin{itemize}
    \item $u_2=u_k$ for $3 \le k \le n$,
    \item $u_2=u_k^{-1}$ for $3 \le k \le n$,
    \item $u_2=(qu_k)^{-1}$ for $3 \le k \le n$.
\end{itemize}

\emph{Case $u_2=u_k$, $3 \le k \le n$:} Both sides of the identity vanish because two rows and their respective columns of the antisymmetric matrix agree, hence the Pfaffian vanishes by \eqref{eq: Antisymmetry of Pfaffian}.

\emph{Case $u_2=u_k^{-1}$, $3 \le k \le n$:} All entries in the second column and in the second row of $\overline{\Mgamma}$ and $\overline{\Mgammarootqbracket_{[n]\setminus\{1\}}}|_{s=qs}$ vanish under this substitution, except for the entries indexed by $k$. It can be checked that these non-vanishing entries of both sides agree. After expanding both sides along the second column and the second row using \eqref{eq: Laplace for Pfaffian},
and cancelling the common prefactor coming from the entries index by $k$ and the sign, we are left with showing that
\begin{multline*}
    (1+q^{1/2})(1+s)\prod_{i=2}^n(1-su_i)\pf \left(\left(\overline{\Mgamma}|_{u_1=s}\right)_{[n]\setminus\{2,k\}}\right)=(1+\gamma q^{1/2})(1+\gamma^{-1}s)\\\times(1-s)\prod_{i=2}^n(s-u_i)(1-su_i) (1-squ_i)\pf\left(\left(\overline{\Mgammarootqbracket_{[n]\setminus\{1\}}}|_{s=qs}\right)_{[n]\setminus\{1,2,k\}}\right),
\end{multline*}
under our current substitution $u_2=u_k^{-1}$. Invoking \eqref{eq: Commutation relation} and cancelling common factors, we see that this is equivalent to 
\begin{multline*}
    (1+q^{1/2})(1+s)\prod_{3\le i\le n\atop i\neq k}(1-su_i)\pf\left(\overline{\Mgammabracket_{[n]\setminus\{2,k\}}}|_{u_1=s}\right)\\=(1+\gamma q^{1/2})(1+\gamma^{-1}s)\prod_{3\le i\le n\atop i\neq k}(s-u_i)(1-su_i) (1-squ_i)\\\times \pf\left(\overline{\Mgammarootqbracket_{[n]\setminus\{1,2,k\}}}|_{s=qs}\right),
\end{multline*}
which holds by induction. 

\emph{Case $u_2=(qu_k)^{-1}$, $3 \le k \le n$:} Again all entries in the second column and in the second row of $\overline{\Mgamma}$ and $\overline{\Mgammarootqbracket_{[n]\setminus\{1\}}}|_{s=qs}$ vanish under this substitution, except for the entries indexed by $k$. It can be checked that the non-vanishing entry of the left-hand side differs from the one on the right-hand side only by the factor 
$$(s-u_k)(1-qsu_k)(q-su_k)^{-1}(1-su_k)^{-1}.$$ After expanding both sides along the second column and the second row using \eqref{eq: Laplace for Pfaffian}
and cancelling the common terms, we are left with showing that
\begin{multline*}
    (1+q^{1/2})(1+s)\frac{(s-u_k)(1-qsu_k)}{(q-su_k)(1-su_k)}\prod_{i=2}^n(1-su_i)\pf \left(\left(\overline{\Mgamma}|_{u_1=s}\right)_{[n]\setminus\{2,k\}}\right)\\=(1+\gamma q^{1/2})(1+\gamma^{-1}s)(1-s)\prod_{i=2}^n(s-u_i)(1-su_i) (1-squ_i)\\\times\pf\left(\left(\overline{\Mgammarootqbracket_{[n]\setminus\{1\}}}|_{s=qs}\right)_{[n]\setminus\{1,2,k\}}\right),
\end{multline*}
under our current substitution $u_2=(qu_k)^{-1}$. Invoking \eqref{eq: Commutation relation}
and after cancelling some factors, this is equivalent to
\begin{multline*}
    (1+q^{1/2})(1+s)\prod_{3\le i\le n\atop i\neq k}(1-su_i)\pf\left(\overline{\Mgammabracket_{[n]\setminus\{2,k\}}}|_{u_1=s}\right)\\=(1+\gamma q^{1/2})(1+\gamma^{-1}s)\prod_{3\le i\le n\atop i\neq k}(s-u_i)(1-su_i) (1-squ_i)\\\times \pf\left(\overline{\Mgammarootqbracket_{[n]\setminus\{1,2,k\}}}|_{s=qs}\right),
\end{multline*}
which holds by induction. 

Now we prove \eqref{eq: key Lemma 2}:
The proof is again by induction with respect to $n$. The case $n=0$ is obvious. 

We now argue that, for $n>0$, both sides of the identity 
are polynomials of degree at most $2n-1$ in $u_i$ for $i=1,\dots,n$. By symmetry, it suffices to show the assertion on the degree for $u_1$. Let $T\subseteq [n]$  and observe that $u_1$ appears only in the column and row of $\overline{\Mgammabracket_T}$ indexed by $1$. It can be checked that the entry of this matrix is a polynomial in $u_1$ of degree $2\vert T\vert -2$ if $1\in T$ and of degree zero otherwise. Now, the claim about the degree is obvious for the left-hand side. Keeping in mind that $M=\Mone$ and a case distinction on whether $1$ is an element of $T$ or not, as well as a simple count of factors in both cases, shows the assertion for the right-hand side.

Next, we identify $2n-1$ evaluation for $u_1$ on which both sides of \eqref{eq: key Lemma 2} agree. More concretely, these are given by
\begin{itemize}
    \item $u_1=u_k$ for $2 \le k \le  n$,
    \item $u_1=u_k^{-1}$ for $2 \le k \le  n$,
    \item $u_1=s$.
\end{itemize}
By the symmetry in $u_1,\ldots,u_n$ and the degree estimation, it then follows that the difference of the left-hand side and the right-hand side is of the form 
$$
C \cdot \prod_{i=1}^n (u_i-s) \prod_{1 \le i < j \le n} (u_i-u_j)(1-u_i u_j), 
$$
where $C \in \mathbb{Q}(\gamma,q^{1/2},s)$. We then also show that both sides agree after setting $(u_1,u_2)=(1,q^{-1})$, thus we can conclude that $C=0$. 

\emph{Case $u_1=u_k$, $2 \le k \le  n$:} The left-hand side of the identity vanishes because two rows and their respective columns of the antisymmetric matrix agree, hence the Pfaffian vanishes by \eqref{eq: Antisymmetry of Pfaffian}, and it remains to show that this is also the case for the right-hand side. 

A summand of the right-hand side vanishes if either $k, n \in T$ or $k, n \in T^c$ because of the antisymmetry of the Pfaffian and the product factors in each summand. Therefore, we can assume that $k \in T$ and $n \in T^c$, or 
$k \in T^c$ or $n \in T$. Assume $T$ is of the first type, then $T'=T \setminus \{k\} \cupdot \{n\}$ is of the second type. It can be checked that 
the summands of $T$ and $T'$ agree under this specialization up to sign using the antisymmetry of the Pfaffian stated in \eqref{eq: Antisymmetry of Pfaffian}. This establishes a sign-reversing involution of the summands under this specialization.

\emph{Case $u_1=u_k^{-1}$, $2 \le k \le n$:} All entries of $\overline{\Mgamma}$ of the first column and of the first row vanish under this substitution, except for the $k$-th entries. Expanding the Pfaffian on the left-hand side along the first row and column by \eqref{eq: Laplace for Pfaffian}
and using \eqref{eq: Commutation relation} with $T=[n]$ and $S=[n]\setminus\{1,k\}$, we obtain
\begin{multline}\label{eq: Simplified LHS eval u=u inverse}
    (-1)^{k}\prod_{2\le l\le n\atop l\neq k}\left(1-\frac{u_l}{u_k}\right)\left(1-q\frac{u_l}{u_k}\right)(1-u_ku_l)(1-qu_ku_l)\\\times\frac{\left(\frac{1}{u_k}-u_k\right)\left((1+q)(1- q^{1 /2}) +\left (\frac{1}{u_k} + u_k\right)(q^{1 /2}- q)\right) }{1+q^{1/2}}\\\times(1+q^{1/2})^2\left(1-\frac{s}{u_k}\right)(1-su_k)\prod_{2\le i\le n\atop i\neq k} (1+q^{1/2})(1-su_i)
    \pf \left(\overline{\Mgammabracket_{[n]\setminus\{1,k\}}}\right).
\end{multline}
Due to the factors $1-u_iu_j$ in the product over all $i\in T,j\in T^c$ and over all $i<j,~i,j\in T^c$, a summand on the right-hand side vanishes under this specialization unless $1,k\in T$. For the summands with $1,k\in T$, we observe that all entries of the column and row indexed by $1$ of $\overline{\Mgammabracket_T}$ vanish except for those indexed by $k$. Expansion of the Pfaffian in each summand along the row and column indexed by $1$ using \eqref{eq: Laplace for Pfaffian} thus yields
\begin{multline*}
    (-1)^{\#\{j\in T| j<k\}} \prod_{k+1\le l\le n\atop l\in T}\left(1-\frac{u_l}{u_k}\right)\left(1-q\frac{u_l}{u_k}\right)(1-u_ku_l)(1-qu_ku_l)\\\times
    \frac{\left(\frac{1}{u_k}-u_k\right)\left((1+q)(1- q^{1 /2}) +\left (\frac{1}{u_k} + u_k\right)(q^{1 /2}- q)\right) }{1+q^{1/2}}
    \prod_{ 2\le l\le k-1\atop l\in T}\left(1-\frac{u_l}{u_k}\right)\left(1-q\frac{u_l}{u_k}\right)\\
    \times\pf\left(\left(\overline{\Mgammabracket_T}|_{u_1=\frac{1}{u_k}}\right)_{T\setminus\{1,k\}}\right),
\end{multline*}
for a summand with $1,k\in T$. Using this and re-indexing the sum on the right-hand side by $S=T\setminus\{1,k\}\subseteq [n]\setminus\{1,k\}$, and after invoking \eqref{eq: Commutation relation} and regrouping terms, we obtain
\begin{multline*}
    (-1)^{k}\prod_{2\le l\le n\atop l\neq k}\left(1-\frac{u_l}{u_k}\right)\left(1-q\frac{u_l}{u_k}\right)(1-u_ku_l)(1-qu_ku_l)\\\times\frac{\left(\frac{1}{u_k}-u_k\right)\left((1+q)(1- q^{1 /2}) +\left (\frac{1}{u_k} + u_k\right)(q^{1 /2}- q)\right) }{1+q^{1/2}}\\\times(1+q^{1/2})^2\left(1-\frac{s}{u_k}\right)(1-su_k)\prod_{2\le i\le n\atop i\neq k} (1+q^{1/2})(1-su_i)\\\times\sum_{S\subseteq[n]\setminus\{1,k\}} \sgn_{[n] \setminus \{1,k \}}(S) (-\gamma^{-1} s;q^{1 /2})_{n-2-\vert S\vert }(-\gamma q^{1/2};q^{1/2})_{n-2-\vert S\vert }\\\times\prod_{i\in S,j\in S^c}(u_{i}-qu_{j})(1-u_iu_j) \prod_{i\in S^c}(1-u_i)\prod_{i\in S}(1+q^{1/2})(u_{i}-s)\\\times\prod_{i,j\in S^c\atop i<j}(1-u_iu_j)(u_i-u_j)\pf\left(\overline{\Mgammabracket_{S}}\right), 
\end{multline*}
where $S^c$ is the complement of $S$ in $[n]\setminus\{1,k\}$.

By induction, this is further equal to 
\begin{multline}
    (-1)^{k}\prod_{2\le l\le n\atop l\neq k}\left(1-\frac{u_l}{u_k}\right)\left(1-q\frac{u_l}{u_k}\right)(1-u_ku_l)(1-qu_ku_l)\\\times\frac{\left(\frac{1}{u_k}-u_k\right)\left((1+q)(1- q^{1 /2}) +\left (\frac{1}{u_k} + u_k\right)(q^{1 /2}- q)\right) }{1+q^{1/2}}\\\times(1+q^{1/2})^2\left(1-\frac{s}{u_k}\right)(1-su_k)\prod_{2\le i\le n\atop i\neq k} (1+q^{1/2})(1-su_i)
    \pf \left(\overline{\Mgammabracket_{[n]\setminus\{1,k\}}}\right)
\end{multline}
and this is the same as \eqref{eq: Simplified LHS eval u=u inverse}.

\emph{Case $u_1=s$:} The left-hand side in this case is given by \begin{equation*}
    (1+q^{1/2})(1-s^2)\prod_{i=2}^n (1+q^{1/2})(1-su_i)
    \pf\left(\overline{\Mgamma}|_{u_1=s}\right).
\end{equation*}
Due to the factors $u_i-s$ in the product over all $i\in T$ on the right-hand side of the identity, each summand with $1\in T$ vanishes. Re-indexing the sum by $S=T\setminus\{1\}\subseteq [n]\setminus\{1\}$ and denoting $S^c=([n]\setminus\{1\})\setminus S$, as well as splitting off the ``lowest order terms'' $1+\gamma q^{1/2}$ and $1+\gamma^{-1}s$ of the Pochhammer-symbols, we obtain for the right-hand side
\begin{multline*}
    (1+\gamma q^{1/2})(1+\gamma s)(1-s)\prod_{i=2}^n(s-u_i)(1-su_i)\sum_{S\subseteq[n]\setminus\{1\}} \sgn_{[n] \setminus \{1\}}(S) \\\times (-\gamma q;q^{1 / 2})_{n-1-\vert S\vert} (-\gamma^{-1}sq^{1/2};q^{1 / 2})_{n-1-\vert S\vert}\prod_{i\in S,j\in S^c}(u_{i}-qu_{j})(1-u_iu_j)\\\times\prod_{i\in S^c}(1-u_i)\prod_{i\in S}(1+q^{1/2})(u_{i}-qs)\prod_{i,j\in S^c\atop i<j}(1-u_iu_j)(u_i-u_j)\pf\left(\overline{M_S}\right).
\end{multline*}
Using the induction hypothesis for $n-1$, $\gamma=\gamma q^{1/2}$ and $s=sq$, we see that this is equal to 
\begin{multline*}
    (1+\gamma q^{1/2})(1+\gamma^{-1}s)(1-s)\prod_{i=2}^n(s-u_i)(1-su_i)\\\times\prod_{i=2}^n (1+q^{1/2})(1-squ_i) \pf\left(\overline{\Mgammarootqbracket_{[n]\setminus\{1\}}}|_{s=qs}\right).
\end{multline*}
Cancelling some common factors, we arrive at the identity \eqref{eq: Key Lemma main2 identity A} in Lemma~\ref{lem: Key lemma for main2}.


\emph{Case $u_1=1$, $u_2=q^{-1}$:} 
We start with the right-hand side. Because of the factor $\prod_{i\in T^c}(1-u_i)$, any summand with $1\in T^c$ vanishes. We re-index the sum by $S=T\setminus\{1\}$ and write the right-hand side as
\begin{multline*}
    (1+q^{1/2})(1-s)\sum_{S\subseteq[n]\setminus\{1\}} \sgn_{[n]\setminus\{1\}}(S)  (-\gamma^{-1} s;q^{1 /2})_{n-1-\vert S\vert }(-\gamma q^{1 /2};q^{1 /2})_{n-1-\vert S\vert }\\\times \prod_{j\in S^c}(1-u_j)^2(1-qu_j)\prod_{i\in S,j\in S^c}(u_{i}-qu_{j})(1-u_iu_j)\prod_{i\in S}(1+q^{1/2})(u_{i}-s)\\\times\prod_{i,j\in S^c\atop i<j}(1-u_iu_j)(u_i-u_j) \pf\left(\overline{M_{S\cup \{1\}}}|_{u_1=1}\right),
\end{multline*}
where we still set $u_2=q^{-1}$ in the full expression and denote  by $S^c$ the complement of $S$ in $[n]\setminus\{1\}$. Due to the factor $\prod_{j\in S^c}(1-qu_j)$, any summand with $2\in S^c$ vanishes under this substitution. It is easy to check that $\left(M_{S_\cup\{1\}}\right)_{0,1}|_{u_1=1}=1$ and $\left(M_{S_\cup\{1\}}\right)_{1,k}|_{u_1=1}=1$ for $k=2,\dots ,n$, therefore any entry in the row and column of $\overline{M_{S\cup \{1\}}}|_{u_1=1}$ indexed by $1$ contains the factor $\prod_{l=2,~l\in S\cup\{1\}}^n(1-qu_l)$ and thus vanishes if $2\in S$. As the Pfaffian then vanishes too, we see that all summands are equal to zero under this substitution, hence the right-hand side vanishes. 

For the left-hand side it is easy to check that $\left(\Mgamma\right)_{0,1}|_{u_1=1}=1$ and $\left(\Mgamma\right)_{1,k}|_{u_1=1}$ is a rational function in $u_k$ with denominator equal to $1-s u_k$. Therefore, any entry in the first row and first column of $\overline{\Mgamma}|_{u_1=1}$ contains the factor $\prod_{l=2}^n(1-qu_l)$ and thus vanishes under the substitution $u_2=q^{-1}$. Since a full row and its corresponding column vanish, the Pfaffian vanishes too, and so the left-hand side is equal to zero as well.
\end{proof}

\subsection{Reducing \eqref{to be eased by notation} to Lemma \ref{lem: Key lemma for main2}}\label{subsec: Proof of Corollary main2}
We now reduce \eqref{to be eased by notation} and thus Corollary \ref{main2:cor} to the $\gamma=1$ case of \eqref{eq: key Lemma 2} in Lemma \ref{lem: Key lemma for main2}.

\begin{proof}[Proof of Corollary \ref{main2:cor}]
    Using our notation introduced in \eqref{eq: Def of Mgamma}, \eqref{eq: Def of B(T)}, \eqref{eq: Notation for matrix restricted 1}, \eqref{eq: Def of Conj by B(T)}, we can rewrite \eqref{to be eased by notation} as
\begin{multline}\label{eq:B}
 \prod_{i=1}^n \frac{1+q^{1/2}}{1-u_i}
    \prod_{1 \le i < j \le n} \frac{1-q  u_i u_j}{u_i-u_j}
    \pf \left(M \right)\\= \sum_{l\ge 0}\sum_{T\subsetneq[n]}
    (-q^{1 /2};q^{1 /2})_{n-\vert T\vert} 
    (-s_l;q^{1 /2})_{n-\vert T\vert }\prod_{i\in T}\frac{(1+q^{1/2})(u_{i}-s_l)}{1-u_i} \\\times\prod_{i=1}^n\left(\frac{1}{1-s_lu_i} \prod_{j=0}^{l-1}\frac{u_i-s_j}{1-s_ju_i}\right) \prod_{i\in T,j\in T^c}\frac{u_{i}-qu_{j}}{u_{i}-u_{j}}
    \prod_{i,j\in T\atop i<j} \frac{1-q  u_i u_j}{u_i-u_j} \pf \left(M_T \right).
\end{multline}
Recall that we need to show this for any fixed $p \ge 0$, where we set $s_l=s_p$ for $l \ge p$. We use the same trick as in the proof of Theorem \ref{main1}. We split the sum over $l$ on the right-hand side into the part where $l < p$ and the remainder where $l \ge p$. The sum over $l \ge p$ can then be computed using the formula for the geometric series. After further multiplying both sides by $1-\prod_{i=1}^n \frac{u_i-s_p}{1- s_p u_i}$, we see that \eqref{eq:B} is equivalent to the following identity.
\begin{multline}\label{B''}
    \left( 1-\prod_{i=1}^n \frac{u_i-s_p}{1- s_p u_i} \right)\prod_{i=1}^n \frac{1+q^{1/2}}{1-u_i}
    \prod_{1 \le i < j \le n} \frac{1-q  u_i u_j}{u_i-u_j}
    \pf \left(M \right)\\= \sum_{T\subsetneq[n]}(-q^{1 /2};q^{1 /2})_{n-\vert T\vert }\prod_{i\in T}\frac{1+q^{1/2}}{1-u_i}\prod_{i\in T,j\in T^c}\frac{u_{i}-qu_{j}}{u_{i}-u_{j}}
    \prod_{i,j\in T\atop i<j} \frac{1-q  u_i u_j}{u_i-u_j}\\\times \pf \left(M_T \right) \left( 
    \sum_{l=0}^{p} (-s_l;q^{1/2})_{n-|T|} \prod_{i \in T} (u_i-s_l) 
    \prod_{i=1}^n \left( \frac{1}{1-s_l u_i} \prod_{j=0}^{l-1} \frac{u_i-s_j}{1-s_j u_i} \right) \right. \\
    \left. -\prod_{i=1}^n \frac{u_i-s_p}{1- s_p u_i} \sum_{l=0}^{p-1} (-s_l;q^{1/2})_{n-|T|} \prod_{i \in T} (u_i-s_l) 
    \prod_{i=1}^n \left( \frac{1}{1-s_l u_i} \prod_{j=0}^{l-1} \frac{u_i-s_j}{1-s_j u_i} \right) \right)
\end{multline}

Next we show that \eqref{B''} follows from
\begin{multline}\label{b}
    \prod_{i=1}^n \prod_{j=0}^{l-1} \frac{u_i-s_j}{1-s_j u_i}\left( 1-\prod_{i=1}^n \frac{u_i-s_l}{1- s_l u_i} \right)\prod_{i=1}^n \frac{1+q^{1/2}}{1-u_i}
    \prod_{1 \le i < j \le n} \frac{1-q  u_i u_j}{u_i-u_j}
    \pf \left(M \right)\\= \sum_{T\subsetneq[n]}(-q^{1 /2};q^{1 /2})_{n-\vert T\vert }\prod_{i\in T}\frac{1+q^{1/2}}{1-u_i}\prod_{i\in T,j\in T^c}\frac{u_{i}-qu_{j}}{u_{i}-u_{j}}
    \prod_{i,j\in T\atop i<j} \frac{1-q  u_i u_j}{u_i-u_j}\\\times \pf \left(M_T \right) (-s_l;q^{1/2})_{n-|T|} \prod_{i \in T} (u_i-s_l) 
    \prod_{i=1}^n \left( \frac{1}{1-s_l u_i} \prod_{j=0}^{l-1} \frac{u_i-s_j}{1-s_j u_i} \right),
\end{multline}
where  $l$ is a non-negative integer. Indeed, we obtain \eqref{B''} after summing \eqref{b} over $l \in \{0,\ldots,p\}$ and then subtracting the sum over $l \in \{0,\ldots,p-1\}$ multiplied by $\prod_{i=1}^n \frac{u_i-s_p}{1-s_p u_i}$. For the right-hand side of \eqref{B''}, this is immediate, while for the left-hand side, this follows from telescoping.

We rearrange some factors in \eqref{b} and move
$$
-\prod_{i=1}^n \prod_{j=0}^{l-1} \frac{u_i-s_j}{1-s_j u_i} \prod_{i=1}^n \frac{u_i-s_l}{1- s_l u_i} \prod_{i=1}^n \frac{1+q^{1/2}}{1-u_i}
    \prod_{1 \le i < j \le n} \frac{1-q  u_i u_j}{u_i-u_j}
    \pf \left(M \right)
$$
(which is just the negative of the missing summand for $T=[n]$ on the right-hand)
from the left-hand side to the right-hand side, yielding
\begin{multline}\label{eq: what we reuse later}
    \prod_{i=1}^n \left( \frac{1+q^{1/2}}{1-u_i}\prod_{j=0}^{l-1} \frac{u_i-s_j}{1-s_j u_i} \right)
    \prod_{1 \le i < j \le n} \frac{1-q  u_i u_j}{u_i-u_j}
    \pf \left(M \right)\\= \sum_{T\subseteq[n]}(-s_l;q^{1/2})_{n-|T|}(-q^{1 /2};q^{1 /2})_{n-\vert T\vert } \prod_{i=1}^n \left( \frac{1}{1-s_l u_i} \prod_{j=0}^{l-1} \frac{u_i-s_j}{1-s_j u_i} \right)\\\times\prod_{i\in T}\frac{(1+q^{1/2})(u_i-s_l)}{1-u_i}\prod_{i\in T,j\in T^c}\frac{u_{i}-qu_{j}}{u_{i}-u_{j}}
    \prod_{i,j\in T\atop i<j} \frac{1-q  u_i u_j}{u_i-u_j} \pf \left(M_T \right).
\end{multline}
We multiply the equation by $\prod_{1\le i<j\le n}(1-u_iu_j)$, then use \eqref{eq:det of B(T)} to form $\det(B([n],[n])$ and $\det(B(T,T))$ on the 
left-hand side and right-hand side, respectively, then invoke \eqref{eq:Pfaffian of conj by B(T)} and divide by $\prod_{i=1}^n \Big(\frac{1}{1-s_lu_i}\prod_{j=0}^{l-1}\frac{u_i-s_j}{1-s_ju_i}\Big)$ to obtain
\begin{multline}\label{eq: Second identification}
    \prod_{i=1}^n \frac{(1+q^{1/2})(1-s_lu_i)}{1-u_i}
    \prod_{1 \le i < j \le n} \frac{1}{u_i-u_j}
    \pf \left(\overline{M} \right)\\= \sum_{T\subseteq[n]}(-s_l;q^{1/2})_{n-|T|}(-q^{1 /2};q^{1 /2})_{n-\vert T\vert }\prod_{i\in T}\frac{(1+q^{1/2})(u_i-s_l)}{1-u_i}\\\times\prod_{i\in T,j\in T^c}\frac{(u_{i}-qu_{j})(1-u_iu_j)}{u_{i}-u_{j}} \prod_{i,j\in T^c}(1-u_iu_j)
    \prod_{i,j\in T\atop i<j} \frac{1}{u_i-u_j} \pf \left(\overline{M_T} \right).
\end{multline}
We then multiply both sides by $\prod_{i=1}^n(1-u_i)\prod_{1\le i<j\le n}(u_i-u_j)$ and replace $s_l$ by $s$. This results in the $\gamma=1$ case of \eqref{eq: key Lemma 2} in Lemma~\ref{lem: Key lemma for main2}. 
\end{proof}

\subsection{Reducing \eqref{eq:to show for main 2} to Lemma \ref{lem: Key lemma for main2}}\label{subsec: Proof of Thm main2}
Using Corollary \ref{main2:cor}, and versions in the proof thereof, we can now reduce \eqref{eq:to show for main 2} to \eqref{eq: key Lemma 2} in Lemma \ref{lem: Key lemma for main2} and thus finish the proof of Theorem \ref{main2}.
\begin{proof}[Proof of Theorem \ref{main2}]
    Rewriting \eqref{eq:to show for main 2} using our notation introduced in \eqref{eq: Def of Mgamma}, where for this proof we substitute $s$ to $s_0$, \eqref{eq: Def of B(T)}, \eqref{eq: Notation for matrix restricted 1}, \eqref{eq: Def of Conj by B(T)}, we obtain
    \begin{multline*}
        \prod_{i=1}^n \frac{1+q^{1/2}}{1-u_i}
    \prod_{1 \le i < j \le n} \frac{1-q  u_i u_j}{u_i-u_j}
    \pf\left(\Mgamma\right)  = \sum_{l \ge 0} \sum_{T \subsetneq [n]} 
    \left([l=0](-\gamma q^{1/ 2};q^{1 / 2})_{n-|T|}\right. \\\times \left.  (-\gamma^{-1}s_0;q^{1 / 2})_{n-|T|}+[l>0](- q^{1/ 2};q^{1 / 2})_{n-|T|} (-s_l;q^{1 / 2})_{n-|T|}\right)  \prod_{i\in T} \frac{(1+q^{1/2})(u_i-s_l)}{1-u_i} \\\times \prod_{i=1}^n \left( \frac{1}{1-s_l u_i} \prod_{j=0}^{l-1} \frac{u_i-s_j}{1-s_j u_i} \right)
    \prod_{i \in T, j \in T^c} \frac{u_i-q u_j}{u_i-u_j}
        \prod_{i,j\in T\atop i<j} \frac{1-q  u_i u_j}{u_i-u_j}
        \pf \left(M_T \right).
    \end{multline*} 
    By adding and subtracting again, we rearrange the $l=0$ term on the right-hand side as follows
    \begin{multline}\label{eq: where they cancel}
    \sum_{l \ge 0} \sum_{T \subsetneq [n]} 
    (- q^{1/ 2};q^{1 / 2})_{n-|T|} (-s_l;q^{1 / 2})_{n-|T|} \prod_{i\in T} \frac{(1+q^{1/2})(u_i-s_l)}{1-u_i} \\\times \prod_{i=1}^n \left( \frac{1}{1-s_l u_i} \prod_{j=0}^{l-1} \frac{u_i-s_j}{1-s_j u_i} \right)
    \prod_{i \in T, j \in T^c} \frac{u_i-q u_j}{u_i-u_j}
        \prod_{i,j\in T\atop i<j} \frac{1-q  u_i u_j}{u_i-u_j}
        \pf \left(M_T \right)\\+\sum_{T \subsetneq [n]} 
    \Big((-\gamma q^{1/ 2};q^{1 / 2})_{n-|T|} (-\gamma^{-1}s_0;q^{1 / 2})_{n-|T|}-(- q^{1/ 2};q^{1 / 2})_{n-|T|} (-s_0;q^{1 / 2})_{n-|T|}\Big) \\\times  \prod_{i=1}^n  \frac{1}{1-s_0 u_i}\prod_{i\in T} \frac{(1+q^{1/2})(u_i-s_0)}{1-u_i} 
    \prod_{i \in T, j \in T^c} \frac{u_i-q u_j}{u_i-u_j}
        \prod_{i,j\in T\atop i<j} \frac{1-q  u_i u_j}{u_i-u_j}
        \pf \left(M_T \right).
    \end{multline}
 We observe that $$(-\gamma q^{1/ 2};q^{1 / 2})_{n-|T|} (-\gamma^{-1}s_0;q^{1 / 2})_{n-|T|}-(- q^{1/ 2};q^{1 / 2})_{n-|T|} (-s_0;q^{1 / 2})_{n-|T|}=0$$ if $T=[n]$, and so we can extend the sum over $T\subsetneq [n]$ in the second summand to a sum over all subsets of $[n]$ and the following partial sum of the second summand on the right-hand side
    \begin{multline*}
        -\sum_{T \subseteq [n]} 
    (- q^{1/ 2};q^{1 / 2})_{n-|T|} (-s_0;q^{1 / 2})_{n-|T|}  \prod_{i=1}^n  \frac{1}{1-s_0 u_i}\prod_{i\in T} \frac{(1+q^{1/2})(u_i-s_0)}{1-u_i} \\\times
    \prod_{i \in T, j \in T^c} \frac{u_i-q u_j}{u_i-u_j}
        \prod_{i,j\in T\atop i<j} \frac{1-q  u_i u_j}{u_i-u_j}
        \pf \left(M_T \right)
    \end{multline*}
    is now, up to the sign, the same as the right-hand side of \eqref{eq: what we reuse later} in the case $l=0$. The sum over $l \ge 0$ in \eqref{eq: where they cancel} is easily identified as the right-hand side of \eqref{eq:B}. Since the right-hand side of \eqref{eq:B} and the right-hand side of \eqref{eq: what we reuse later}, in the case $l=0$, evaluate to the same left-hand side and one of them appears with a negative sign in the above formula \eqref{eq: where they cancel}, they cancel, and we are left with showing that
    \begin{multline}
        \prod_{i=1}^n \frac{1+q^{1/2}}{1-u_i}
    \prod_{1 \le i < j \le n} \frac{1-q  u_i u_j}{u_i-u_j}
    \pf\left(\Mgamma\right) =\sum_{T \subsetneq [n]} 
    (-\gamma q^{1/ 2};q^{1 / 2})_{n-|T|} (-\gamma^{-1}s_0;q^{1 / 2})_{n-|T|} \\\times  \prod_{i=1}^n  \frac{1}{1-s_0 u_i}\prod_{i\in T} \frac{(1+q^{1/2})(u_i-s_0)}{1-u_i} 
    \prod_{i \in T, j \in T^c} \frac{u_i-q u_j}{u_i-u_j}
        \prod_{i,j\in T\atop i<j} \frac{1-q  u_i u_j}{u_i-u_j}
        \pf \left(M_T \right).
    \end{multline}
    Multiplying by $\prod_{i=1}^n(1-u_i)\prod_{1\le i<j\le n}(u_i-u_j)$ and replacing $s_0$ by $s$ results in the identity \eqref{eq: key Lemma 2} stated in Lemma~\ref{lem: Key lemma for main2}.
\end{proof}

\section{Concluding remarks} 

In this paper, we establish two Littlewood identities that generalize known Littlewood identities for generating functions associated with alternating sign matrices. These identities play a central role in computations linking various classes of objects that are equinumerous with alternating sign matrices. Consequently, combinatorial proofs of such identities advance our understanding of how to connect these classes in a purely combinatorial manner. We are currently investigating whether a vertex-model proof of Theorem~\ref{main2}, which is based on the Yang–Baxter equation, we have worked out can be developed into a probabilistic bijection of that Littlewood identity.

\bibliographystyle{abbrvurl}
\bibliography{ASMLittlewood.bib}

@book {EC2,
    AUTHOR = {Stanley, Richard P.},
     TITLE = {Enumerative combinatorics. {V}ol. 2},
    SERIES = {Cambridge Studies in Advanced Mathematics},
    VOLUME = {62},
      NOTE = {With a foreword by Gian-Carlo Rota and appendix 1 by Sergey
              Fomin},
 PUBLISHER = {Cambridge University Press, Cambridge},
      YEAR = {1999},
     PAGES = {xii+581},
      ISBN = {0-521-56069-1; 0-521-78987-7},
   MRCLASS = {05A15 (05-02 05E05 05E10 68R05)},
  MRNUMBER = {1676282},
MRREVIEWER = {Ira\ Gessel},
       DOI = {10.1017/CBO9780511609589},
       URL = {https://doi.org/10.1017/CBO9780511609589},
}

@ARTICLE{Fis19a,
AUTHOR = {Fischer, I.},
TITLE = {Enumeration of alternating sign triangles using a constant term approach},
JOURNAL = {Trans. Amer. Math. Soc.},
VOLUME = {372},
YEAR = {2019},
PAGES = {1485--1508},
}

@book {Macdonald,
    AUTHOR = {Macdonald, I. G.},
     TITLE = {Symmetric functions and orthogonal polynomials},
    SERIES = {University Lecture Series},
    VOLUME = {12},
      NOTE = {Dean Jacqueline B. Lewis Memorial Lectures presented at
              Rutgers University, New Brunswick, NJ},
 PUBLISHER = {American Mathematical Society, Providence, RI},
      YEAR = {1998},
     PAGES = {xvi+53},
      ISBN = {0-8218-0770-6},
   MRCLASS = {05E05 (05E35 33D80)},
  MRNUMBER = {1488699},
MRREVIEWER = {Ang\`ele\ M.\ Hamel},
       DOI = {10.1090/ulect/012},
       URL = {https://doi.org/10.1090/ulect/012},
}

@ARTICLE{Fis19b,
AUTHOR = {Fischer, I.},
TITLE = {A constant term approach to enumerating alternating sign trapezoids},
JOURNAL = {Adv. Math.},
VOLUME = {356},
YEAR = {2019},
PAGES = {},
}

@article {AlternatingSignPentagonsAndMagogPentagons,
    AUTHOR = {Gangl, Moritz},
     TITLE = {Alternating sign pentagons and {M}agog pentagons},
   JOURNAL = {Adv. Math.},
  FJOURNAL = {Advances in Mathematics},
    VOLUME = {474},
      YEAR = {2025},
     PAGES = {Paper No. 110315, 29},
      ISSN = {0001-8708,1090-2082},
   MRCLASS = {05B20 (05A15)},
  MRNUMBER = {4903258},
MRREVIEWER = {Miodrag\ \v{Z}ivkovi\'{c}},
       DOI = {10.1016/j.aim.2025.110315},
       URL = {https://doi.org/10.1016/j.aim.2025.110315},
}

@article {fourfold,
    AUTHOR = {H{\"o}ngesberg, H.},
     TITLE = {A fourfold refined enumeration of alternating sign trapezoids},
   JOURNAL = {Electron. J. Combin.},
  FJOURNAL = {Electronic Journal of Combinatorics},
    VOLUME = {29},
      YEAR = {2022},
    NUMBER = {3},
     PAGES = {Paper No. 3.42, 27},
   MRCLASS = {05A15 (05A05 15B35)},
  MRNUMBER = {4477841},
       DOI = {10.37236/9933},
}

@misc{bounded,
AUTHOR = {Fischer, I.},
      EPRINT = {2301.00175},
  EPRINTTYPE = {arxiv},
TITLE = {Bounded Littlewood identity related to alternating sign matrices},
YEAR = {to appear in \it Forum of Mathematics, Sigma}
}

@BOOK{Bre99,
AUTHOR = {Bressoud, D. M.},
TITLE = {Proofs and confirmations. {T}he story of the alternating sign matrix conjecture},
SERIES = {MAA Spectrum},
PUBLISHER = {Mathematical Association of America and Cambridge University Press},
ADDRESS = {Washington, DC and Cambridge},
YEAR = {1999},
PAGES = {xvi+274},
MRNUMBER = {1718370 (2000i:15002)},
}

@article {nASMDPP,
    AUTHOR = {Fischer, I. and Schreier-Aigner, F.},
     TITLE = {The relation between alternating sign matrices and descending
              plane partitions: {$n + 3$} pairs of equivalent statistics},
   JOURNAL = {Adv. Math.},
  FJOURNAL = {Advances in Mathematics},
    VOLUME = {413},
      YEAR = {2023},
     PAGES = {Paper No. 108831, 47},
   MRCLASS = {05A15 (05A05 05A19 15B35 82B20 82B23)},
  MRNUMBER = {4530620},
MRREVIEWER = {David\ M.\ Bressoud},
       DOI = {10.1016/j.aim.2022.108831},
}

@incollection {kirillov,
    AUTHOR = {Kirillov, A. N. and Noumi, M.},
     TITLE = {{$q$}-difference raising operators for {M}acdonald polynomials
              and the integrality of transition coefficients},
 BOOKTITLE = {Algebraic methods and {$q$}-special functions ({M}ontr\'eal,
              {QC}, 1996)},
    SERIES = {CRM Proc. Lecture Notes},
    VOLUME = {22},
     PAGES = {227--243},
 PUBLISHER = {Amer. Math. Soc., Providence, RI},
      YEAR = {1999},
   MRCLASS = {39A13 (05E05 33D52)},
  MRNUMBER = {1726838},
MRREVIEWER = {Laurent\ Habsieger},
       DOI = {10.1090/crmp/022/13},
}

@article {warnaar,
    AUTHOR = {Warnaar, S. O.},
     TITLE = {Bisymmetric functions, {M}acdonald polynomials and
              {$\mathfrak{sl}_3$} basic hypergeometric series},
   JOURNAL = {Compos. Math.},
  FJOURNAL = {Compositio Mathematica},
    VOLUME = {144},
      YEAR = {2008},
    NUMBER = {2},
     PAGES = {271--303},
   MRCLASS = {33D67 (05A30 05E05)},
  MRNUMBER = {2406113},
MRREVIEWER = {Laurent\ Habsieger},
       DOI = {10.1112/S0010437X07003211},
}

@article {Petrov,
    AUTHOR = {Petrov, Leonid},
     TITLE = {Refined {C}auchy identity for spin {H}all-{L}ittlewood
              symmetric rational functions},
   JOURNAL = {J. Combin. Theory Ser. A},
  FJOURNAL = {Journal of Combinatorial Theory. Series A},
    VOLUME = {184},
      YEAR = {2021},
     PAGES = {Paper No. 105519, 50},
   MRCLASS = {05E05 (05E16 16T25 22E45)},
  MRNUMBER = {4297032},
MRREVIEWER = {Elizabeth\ M.\ Niese},
       DOI = {10.1016/j.jcta.2021.105519},
}

@article {Gavrilova,
    AUTHOR = {Gavrilova, Svetlana},
     TITLE = {Refined {L}ittlewood identity for spin {H}all-{L}ittlewood
              symmetric rational functions},
   JOURNAL = {Algebr. Comb.},
  FJOURNAL = {Algebraic Combinatorics},
    VOLUME = {6},
      YEAR = {2023},
    NUMBER = {1},
     PAGES = {37--51},
      ISSN = {2589-5486},
   MRCLASS = {05E05},
  MRNUMBER = {4552707},
MRREVIEWER = {Shintarou\ Yanagida},
       DOI = {10.5802/alco.251},
       URL = {https://doi.org/10.5802/alco.251},
}

@article {Borodin,
    AUTHOR = {Borodin, Alexei},
     TITLE = {On a family of symmetric rational functions},
   JOURNAL = {Adv. Math.},
  FJOURNAL = {Advances in Mathematics},
    VOLUME = {306},
      YEAR = {2017},
     PAGES = {973--1018},
      ISSN = {0001-8708},
   MRCLASS = {82C22 (05E05 33C52 81P45)},
  MRNUMBER = {3581324},
       DOI = {10.1016/j.aim.2016.10.040},
       URL = {https://doi.org/10.1016/j.aim.2016.10.040},
}

@article {BorodinPetrov,
    AUTHOR = {Borodin, Alexei and Petrov, Leonid},
     TITLE = {Higher spin six vertex model and symmetric rational functions},
   JOURNAL = {Selecta Math. (N.S.)},
  FJOURNAL = {Selecta Mathematica. New Series},
    VOLUME = {24},
      YEAR = {2018},
    NUMBER = {2},
     PAGES = {751--874},
      ISSN = {1022-1824},
   MRCLASS = {60K35 (05E05 82B23)},
  MRNUMBER = {3782413},
       DOI = {10.1007/s00029-016-0301-7},
       URL = {https://doi.org/10.1007/s00029-016-0301-7},
}

@article {KawankasIdentity,
    AUTHOR = {Kawanaka, Noriaki},
     TITLE = {On subfield symmetric spaces over a finite field},
   JOURNAL = {Osaka J. Math.},
  FJOURNAL = {Osaka Journal of Mathematics},
    VOLUME = {28},
      YEAR = {1991},
    NUMBER = {4},
     PAGES = {759--791},
      ISSN = {0030-6126},
   MRCLASS = {20G05 (20G40)},
  MRNUMBER = {1152953},
MRREVIEWER = {Bhama Srinivasan},
       URL = {http://projecteuclid.org/euclid.ojm/1200783420},
}

@book {LittlewoodIdentityOriginal,
    AUTHOR = {Littlewood, Dudley E.},
     TITLE = {The theory of group characters and matrix representations of
              groups},
      NOTE = {Reprint of the second (1950) edition},
 PUBLISHER = {AMS Chelsea Publishing, Providence, RI},
      YEAR = {2006},
     PAGES = {viii+314},
      ISBN = {0-8218-4067-3},
   MRCLASS = {20C15 (20G05)},
  MRNUMBER = {2213154},
       DOI = {10.1090/chel/357},
       URL = {https://doi.org/10.1090/chel/357},
}

@article {Eigenfunctions1,
    AUTHOR = {Povolotsky, A. M.},
     TITLE = {On the integrability of zero-range chipping models with
              factorized steady states},
   JOURNAL = {J. Phys. A},
  FJOURNAL = {Journal of Physics. A. Mathematical and Theoretical},
    VOLUME = {46},
      YEAR = {2013},
    NUMBER = {46},
     PAGES = {465205, 25},
      ISSN = {1751-8113,1751-8121},
   MRCLASS = {81R12 (82C22)},
  MRNUMBER = {3126878},
       DOI = {10.1088/1751-8113/46/46/465205},
       URL = {https://doi.org/10.1088/1751-8113/46/46/465205},
}

@article {Eigenfunctions2,
    AUTHOR = {Borodin, Alexei and Corwin, Ivan and Petrov, Leonid and
              Sasamoto, Tomohiro},
     TITLE = {Spectral theory for interacting particle systems solvable by
              coordinate {B}ethe ansatz},
   JOURNAL = {Comm. Math. Phys.},
  FJOURNAL = {Communications in Mathematical Physics},
    VOLUME = {339},
      YEAR = {2015},
    NUMBER = {3},
     PAGES = {1167--1245},
      ISSN = {0010-3616,1432-0916},
   MRCLASS = {82B23 (81R12)},
  MRNUMBER = {3385995},
       DOI = {10.1007/s00220-015-2424-7},
       URL = {https://doi.org/10.1007/s00220-015-2424-7},
}

@article {Eigenfunctions3,
    AUTHOR = {Corwin, Ivan and Petrov, Leonid},
     TITLE = {Stochastic higher spin vertex models on the line},
   JOURNAL = {Comm. Math. Phys.},
  FJOURNAL = {Communications in Mathematical Physics},
    VOLUME = {343},
      YEAR = {2016},
    NUMBER = {2},
     PAGES = {651--700},
      ISSN = {0010-3616,1432-0916},
   MRCLASS = {82B23 (60K35 82C23)},
  MRNUMBER = {3477349},
MRREVIEWER = {Aernout\ C. D. van Enter},
       DOI = {10.1007/s00220-015-2479-5},
       URL = {https://doi.org/10.1007/s00220-015-2479-5},
}

@article {LWIdentityHistoryFullReferences,
    AUTHOR = {Rains, Eric and Warnaar, S. Ole},
     TITLE = {Bounded {L}ittlewood identities},
   JOURNAL = {Mem. Amer. Math. Soc.},
  FJOURNAL = {Memoirs of the American Mathematical Society},
    VOLUME = {270},
      YEAR = {2021},
    NUMBER = {1317},
     PAGES = {vii+115},
      ISSN = {0065-9266,1947-6221},
      ISBN = {978-1-4704-4690-1; 978-1-4704-6522-3},
   MRCLASS = {05-02 (05A30 05E05 11P84 17B67 33D67)},
  MRNUMBER = {4259867},
       DOI = {10.1090/memo/1317},
       URL = {https://doi.org/10.1090/memo/1317},
}

@article {LWCombAspects3,
    AUTHOR = {Betea, D. and Wheeler, M.},
     TITLE = {Refined {C}auchy and {L}ittlewood identities, plane partitions
              and symmetry classes of alternating sign matrices},
   JOURNAL = {J. Combin. Theory Ser. A},
  FJOURNAL = {Journal of Combinatorial Theory. Series A},
    VOLUME = {137},
      YEAR = {2016},
     PAGES = {126--165},
      ISSN = {0097-3165,1096-0899},
   MRCLASS = {05E05},
  MRNUMBER = {3403518},
MRREVIEWER = {Meesue\ Yoo},
       DOI = {10.1016/j.jcta.2015.08.007},
       URL = {https://doi.org/10.1016/j.jcta.2015.08.007},
}

@article {LWCombAspects4,
    AUTHOR = {Betea, D. and Wheeler, M. and Zinn-Justin, P.},
     TITLE = {Refined {C}auchy/{L}ittlewood identities and six-vertex model
              partition functions: {II}. {P}roofs and new conjectures},
   JOURNAL = {J. Algebraic Combin.},
  FJOURNAL = {Journal of Algebraic Combinatorics. An International Journal},
    VOLUME = {42},
      YEAR = {2015},
    NUMBER = {2},
     PAGES = {555--603},
      ISSN = {0925-9899,1572-9192},
   MRCLASS = {05E05},
  MRNUMBER = {3369568},
MRREVIEWER = {Trueman\ MacHenry},
       DOI = {10.1007/s10801-015-0592-3},
       URL = {https://doi.org/10.1007/s10801-015-0592-3},
}

@article {LWCombAspects13,
    AUTHOR = {Wheeler, Michael and Zinn-Justin, Paul},
     TITLE = {Refined {C}auchy/{L}ittlewood identities and six-vertex model
              partition functions: {III}. {D}eformed bosons},
   JOURNAL = {Adv. Math.},
  FJOURNAL = {Advances in Mathematics},
    VOLUME = {299},
      YEAR = {2016},
     PAGES = {543--600},
      ISSN = {0001-8708,1090-2082},
   MRCLASS = {05E05},
  MRNUMBER = {3519476},
MRREVIEWER = {Michael\ Orin\ Joyce},
       DOI = {10.1016/j.aim.2016.05.010},
       URL = {https://doi.org/10.1016/j.aim.2016.05.010},
}

@book {schur1,
    AUTHOR = {Schur, Issai},
     TITLE = {Gesammelte {A}bhandlungen. {B}and {III}},
      NOTE = {Herausgegeben von Alfred Brauer und Hans Rohrbach},
 PUBLISHER = {Springer-Verlag, Berlin-New York},
      YEAR = {1973},
     PAGES = {iv+480},
   MRCLASS = {01A75},
  MRNUMBER = {462893},
}

@article {schur2,
    AUTHOR = {Schur, Issai},
     TITLE = {Aufgabe 569},
   JOURNAL = {Arch. Math. Phys.},
    VOLUME = {27(3)},
      YEAR = {1998},
}

@misc{FHRobbins,
AUTHOR = {Fischer, I. and H{\"o}ngesberg, H.},
      EPRINT = {2505.09275},
  EPRINTTYPE = {arxiv},
TITLE = {A {L}ittlewood-type identity for {R}obbins polynomials},
YEAR = {to appear in Algebraic Combinatorics}
}

\end{document}